\documentclass[11pt,a4paper]{amsart}

\usepackage{color}
\usepackage{amsmath, enumerate}
\usepackage{amsfonts}
\usepackage{amsthm}
\usepackage{mathrsfs}
\usepackage{amssymb}
\usepackage[english]{babel}
\usepackage{graphicx}
\usepackage[all]{xy}
\usepackage{yfonts}

\usepackage[active]{srcltx}
\usepackage[parfill]{parskip}

\newtheorem{theorem}{Theorem}
\newtheorem{lemma}[theorem]{Lemma}

\newtheorem{corollary}[theorem]{Corollary}
\newtheorem{proposition}[theorem]{Proposition}
\newtheorem{remark}[theorem]{Remark}
\newtheorem{example}[theorem]{Example}
\newtheorem{conjecture}[theorem]{Conjecture}
\newtheorem{question}[theorem]{Question}

\newcommand{\Ev}{{\mathrm{Ev}}}

\newcommand{\add}{{\mathrm{add}}}

\newcommand{\eins}{\leavevmode\hbox{\small1\kern-3.8pt\normalsize1}}

\newcommand{\tto}{\twoheadrightarrow}

\newcommand{\mC}{\mathbb{C}}

\newcommand{\mZ}{\mathbb{Z}}

\newcommand{\fk}{{\mathfrak k}}

\font\sc=rsfs10
\newcommand{\cC}{\sc\mbox{C}\hspace{1.0pt}}

\newcommand{\cI}{\sc\mbox{I}\hspace{1.0pt}}

\newcommand{\cL}{\sc\mbox{L}\hspace{1.0pt}}
\newcommand{\cP}{\sc\mbox{P}\hspace{1.0pt}}

\newcommand{\End}{{\rm End}}
\newcommand{\Id}{{\rm Id}}

\begin{document}

\title[manipulations with simple modules]
{Indecomposable manipulations with \\simple modules in category $\mathcal{O}$}

\author{Kevin Coulembier, Volodymyr Mazorchuk and Xiaoting Zhang}
\date{}

\begin{abstract}
We study the problem of indecomposability of translations 
of simple modules in the principal block of BGG category $\mathcal{O}$ for $\mathfrak{sl}_n$, 
as conjectured in \cite{KiM}. We describe some general techniques and prove 
a few general results which may be
applied to study various special cases of this problem.
We apply our results to verify indecomposability for $n\leq 6$.
We also study the problem of indecomposability of shufflings and twistings of simple modules
and obtain some partial results.
\end{abstract}

\maketitle



\section{Introduction}\label{s1}

Consider the simple complex Lie algebra $\mathfrak{sl}_n$, for $n\geq 2$,
with a fixed standard triangular decomposition $\mathfrak{sl}_n=\mathfrak{n}_-\oplus \mathfrak{h}\oplus
\mathfrak{n}_+$ and let $\mathcal{O}$ be the associated BGG category $\mathcal{O}$ as in \cite{BGG,Humphreys}.
The Weyl group $W$ of $\mathfrak{sl}_n$ indexes naturally both simple modules $L(w)$, where $w\in W$, in the
principal block $\mathcal{O}_0$ of $\mathcal{O}$ and indecomposable projective functors
$\theta_w$, where $w\in W$, in $\mathcal{O}_0$, see \cite{BG}. The following conjecture is
formulated in \cite[Conjecture~2]{KiM}.

\begin{conjecture}\label{conj1}
For all $x,y\in W$, the module $\theta_x\,L(y)$ is either indecomposable or zero.
\end{conjecture}

The present paper arose as a result of various, so far unsuccessful, attempts to either prove or disprove
this conjecture. If true, the statement of Conjecture~\ref{conj1} would be a type A phenomenon, as it
is well-known that it fails already in type $B_2$, see e.g. \cite[Subsection~5.1]{KiM}.

As mentioned in \cite{KiM}, the question posed by Conjecture~\ref{conj1} has a history
related to classification of projective functors on parabolic category $\mathcal{O}$ in type A.
That problem was solved in \cite{KiM} using results from higher representation theory, namely, classification
of simple transitive $2$-representations of the $2$-category of Soergel bimodules over the
coinvariant algebra of $W$, see \cite{MM5}. The latter motivates one of the approaches which
we try in the present paper. The other approaches are more ``classical'' and directed towards
understanding the endomorphisms algebra of $\theta_x\,L(y)$, by comparing it, using homological 
methods, with the endomorphism algebra of some other indecomposable modules.

The paper is organized as follows: In Section~\ref{s2} we collected all preliminaries 
and generalities on category $\mathcal{O}$. In Section~\ref{s3} we describe some special
cases of Conjecture~\ref{conj1} and also recall the results and observations from \cite{KiM}.
In Section~\ref{s4} we describe two homological approaches to Conjecture~\ref{conj1}. One of them,
presented in Subsection~\ref{s4.1}, compares the endomorphism algebra $\theta_x\,L(y)$ to that
of a (twisted) indecomposable projective module in $\mathcal{O}_0$. The other one, 
presented in Subsection~\ref{s4.2}, explores the possibility of an inductive proof with respect
to the Bruhat order. Section~\ref{s7} outlines a higher representation theoretic approach which 
is based on understanding the action on $\mathcal{O}_0$ of the $2$-full $2$-subcategory of the
$2$-category of projective functors generated by $\theta_x$. None of the above approaches seem
to work in full generality for the moment. However, a common feature of the above approaches is that
they all work in the special case of $x$ being the longest element of a parabolic subalgebra
(this case is ``easy'' and was already described in \cite{KiM}). All the approaches mentioned
about motivate a number of question, of homological or higher representation theoretical nature, 
respectively, which we emphasize and study (or describe the answer to) in special case.

Section~\ref{s5} contains two general observations which, in particular, turn out to be useful
for the study of Conjecture~\ref{conj1}. We observe a ``recursive'' structure of $\mathcal{O}_0$
in the sense that, for any parabolic subgroup of $W$, the category $\mathcal{O}_0$ turns out to have 
a filtration by Serre subcategories such that the corresponding subquotients are equivalent to 
the category $\mathcal{O}_0$ considered with respect to the Lie algebra associated with the parabolic 
subgroup. We show two ways to construct such a filtration (using left, respectively, right cosets
of $W$ with respect to the parabolic subgroup). One of this ways is well-coordinated with the 
action of twisting functors while the other one is well-coordinated with the action of projective functors.
The latter in particular, implies that, for fixed $x$ and $y$, Conjecture~\ref{conj1} has the same answer 
for all $n$ for which the elements $x$ and $y$ make sense (i.e. may be defined using 
the obvious chain of inclusions for symmetric groups). Another consequence is that if,
for fixed $x$ and $n$, Conjecture~\ref{conj1} is true for all $y$, then it is true,
for the same $x$, for all $n$ and for all $y$ for which it can be formulated.

In Section~\ref{s8}, we prove Conjecture~\ref{conj1} in the cases $n=2,3,4,5,6$.
The cases $n=2,3,4$ are easy and the case $n=5$ can be dealt with using the results from 
Section~\ref{s5}. The case $n=6$ is substantially more difficult and requires a number of case-by-case studies.
After application of all general methods which we know (in particular, of those from Section~\ref{s5}),
for $n=6$, we are left with five cases which we have to go through on a case-by-case basis and using
various tricks and explicit computations with Kazhdan-Lusztig polynomials (for this we use explicit
tables from \cite{Go}). So far we did not manage to extend the arguments we used in these specific
cases to any more general situation. However, the versatility of difficulties that show up 
for $n=6$ suggests that  the general case of Conjecture~\ref{conj1} might be really difficult.

Finally, in  Section~\ref{s9} we study the problem of indecomposability of shufflings of simple modules in category
$\mathcal{O}$. We show that it is equivalent to the problem of indecomposability of twistings 
of simple modules in category $\mathcal{O}$ and also establish this indecomposability in a number of cases.
We conjecture that any shuffling (or twisting) of a simple module in category $\mathcal{O}$ in any Lie type
is either indecomposable or zero.
\vspace{0.5cm}

{\bf Acknowledgements.}
The first author is supported by the Australian Research Council.
The second and the third authors are supported by the Swedish Research Council and
G{\"o}ran Gustafsson Foundation.
\vspace{2cm}

\section{Category $\mathcal{O}$}\label{s2}

\subsection{The Lie algebra and weights}\label{s2.1}
We work over the ground field $\mathbb{C}$ of complex numbers. Some of our results will be
also valid outside type A. Therefore we introduce all main objects in wider generality and
then, eventually, restrict to type A when it is necessary.

Let $\mathfrak{g}$ be a {\em reductive Lie algebra} with a fixed {\em triangular decomposition}
\begin{equation}\label{eq1}
\mathfrak{g}\;=\;\mathfrak{n}^-\oplus\mathfrak{h}\oplus\mathfrak{n}^+.
\end{equation}
Here $\mathfrak{h}$ is a fixed {\em Cartan subalgebra} and
$\mathfrak{b}=\mathfrak{h}\oplus\mathfrak{n}^+$ is the corresponding {\em Borel subalgebra}.
We denote by $R=R_+\cup R_-$ the corresponding root system of $\mathfrak{g}$ in
$\mathfrak{h}^*$ decomposed into positive and negative roots and
by $W$ the associated Weyl group. 
The half of the sum of all positive roots is denoted by $\rho\in\mathfrak{h}^\ast$
and the $W$-invariant form on $\mathfrak{h}^*$ is denoted $\langle \cdot,\cdot\rangle$.

Let~$\Lambda\subset\mathfrak{h}^\ast$ denote the set of
{\em integral weights}, that is weights which appear in finite dimensional
$\mathfrak{g}$-modules. The {\em natural partial order} on $\Lambda$ is
defined by setting $\lambda\ge \mu$ if and only if $\lambda-\mu$ is a (possibly empty)
sum of positive roots.  The {\em dominant weights} form the subset
\begin{displaymath}
\Lambda^+=\{\lambda\,|\,\langle \lambda+\rho,\alpha\rangle \ge 0,\;\;\mbox{for all $\alpha\in \Delta^+$}\}.
\end{displaymath}
The subset $\Lambda^{++}$ of {\em regular} dominant weights in $\Lambda^{+}$ consists of those
weights for which the above inequality is strict.

For a given {\em parabolic subalgebra} $\mathfrak{p}$ satisfying
$\mathfrak{b}\subset\mathfrak{p}\subset \mathfrak{g}$, we have the corresponding
{\em parabolic decomposition}
\begin{displaymath}
\mathfrak{g}\;=\;\mathfrak{u}^-\oplus \mathfrak{l}\oplus \mathfrak{u}^+,
\end{displaymath}
where $\mathfrak{l}$ is the {\em Levi subalgebra} of $\mathfrak{p}$ and $\mathfrak{u}^+$
is the {\em nilpotent radical} of $\mathfrak{p}$. As $\langle \cdot,\cdot\rangle$ is non-degenerate,
there exists an element $H_{\mathfrak{l}}\in\mathfrak{h}$, such that~$\lambda(H_{\mathfrak{l}})=\langle\lambda,\rho-\rho(\mathfrak{l}) \rangle$, with $\rho(\mathfrak{l})$ being the half of the sum of all positive roots
of $\mathfrak{l}$. It then follows that the action of $\mathrm{ad}_{H_{\mathfrak{l}}}$ on~$\mathfrak{g}$ yields
a $\mZ$-grading where
\begin{displaymath}
\mathfrak{g}_0=\mathfrak{l},\qquad
\mathfrak{u}^-=\bigoplus_{i<0}\mathfrak{g}_i\qquad\text{ and }\qquad
\mathfrak{u}^-=\bigoplus_{i<0}\mathfrak{g}_i.
\end{displaymath}

For a Lie algebra $\mathfrak{a}$, we denote by $U(\mathfrak{a})$ the universal enveloping algebra
of $\mathfrak{a}$. We set $U:=U(\mathfrak{g})$.

\subsection{The Weyl group and parabolic subgroups}\label{s2.2}

The Weyl group $W$ is partially ordered with respect to the {\em Bruhat order} $\preceq$,
see \cite[Section~0.4]{Humphreys}. We choose the convention where the identity element $e$
of $W$ is the minimum element and the {\em longest element} $w_0$ in $W$ is the maximum element.
The {\em length function} on $W$ is denoted by $\ell$ and the set of {\em simple reflections} in $W$
is denoted by $S$.

We consider the {\em dot action} of $W$ on $\mathfrak{h}^*$ given by
$w\cdot\lambda=w(\lambda+\rho)-\rho$, for $w\in W$ and~$\lambda\in\Lambda$.
The stabilizer of $\lambda\in \mathfrak{h}^*$ with respect to this action is denoted by
$W_\lambda\subset W$. Then we let $X_\lambda$ denote the set of longest representatives in
$W$ of the cosets in $W/W_\lambda$.

Given a parabolic subalgebra $\mathfrak{p}$ of $\mathfrak{g}$ as in the previous subsection,
the Weyl group of $\mathfrak{l}$ with respect to $\mathfrak{h}$ is denoted by $W^{\mathfrak{p}}$
and has longest element $w_0^{\mathfrak{p}}$. The set of shortest coset representatives in
$W/W^{\mathfrak{p}}$ is denoted by~$X^{\mathfrak{p}}$ and the corresponding set for
$W^{\mathfrak{p}}\backslash W$ is denoted by~${}^{\mathfrak{p}}X$.


The second partial order on $\Lambda$ only has relations inside  Weyl group
orbits and, for any $\lambda\in\Lambda^+$, we set
\begin{displaymath}
w_1\cdot\lambda\preceq w_2\cdot\lambda\quad\Leftrightarrow\quad w_1
\succeq w_2,\qquad\mbox{for all $w_1,w_2\in X_\lambda$}.
\end{displaymath}
The partial order $\preceq$ is also generated by the relations~$r\cdot\mu\prec\mu$, where $\mu\in\Lambda$
and $r\in W$ is a (not necessarily simple) reflection such that~$r\cdot\mu\le \mu$.

A subset $K\subset\Lambda$ is called {\bf saturated} if it is an ideal for the partial order
$\preceq$. Concretely, for any~$\lambda\in K$ and~$\mu\preceq \lambda$, the requirement is
that~$\mu\in K$.

For $w\in W$, the {\em support} $\mathrm{supp}(w)$ of $w$ is the set of simple reflections
which appear in (any) reduced decomposition of $w$.

\subsection{BGG category $\mathcal{O}$}\label{s2.3}
Consider the {\em BGG category  $\mathcal{O}$} associated to the triangular decomposition \eqref{eq1},
see \cite{BGG, Humphreys, Jantzen}. Simple objects in $\mathcal{O}$ are, up to isomorphism,
{\em simple highest weight modules} $L(\mu)$, where $\mu\in\mathfrak{h}^\ast$. The module
$L(\mu)$ is the simple top of the {\em Verma module}
$\Delta(\mu)$ and has highest weight $\mu$. The projective cover of $L(\mu)$ in $\mathcal{O}$
is denoted $P(\mu)$. The injective envelope of $L(\mu)$ in $\mathcal{O}$
is denoted $I(\mu)$.

We will only consider the {\em integral part} $\mathcal{O}_\Lambda$ of $\mathcal{O}$ which consists of all modules
with weights in $\Lambda$.  The category $\mathcal{O}_\Lambda$ decomposes into {\em indecomposable} blocks
as follows:
\begin{displaymath}
\mathcal{O}_\Lambda\;=\;\bigoplus_{\lambda\in\Lambda^+}\mathcal{O}_\lambda,
\end{displaymath}
where $\mathcal{O}_\lambda$, for $\lambda\in\Lambda^+$, is the Serre subcategory of $\mathcal{O}$ generated by
all simples of the form $L(x\cdot\lambda)$, where $x\in X_\lambda$.
For $\lambda=0$, the corresponding block $\mathcal{O}_0$ is called the
{\em principal block}. Note that the orbit $W\cdot 0$ is regular and hence isomorphism classes of simples
in $\mathcal{O}_0$ are in bijection with elements in $W$. For $w\in W$, we will often denote $L(w\cdot 0)$
simply  by $L(w)$ and similarly for all other structural modules.

By \cite[Theorem~5.1]{Humphreys}, for all $\mu,\nu\in\Lambda$, we have
\begin{equation}\label{VermaBGG}
[\Delta(\mu):L(\nu)]\not=0\quad\Leftrightarrow\quad \nu\preceq\mu\quad\Leftrightarrow\quad
\Delta(\nu)\subset \Delta(\mu).
\end{equation}

For any $K\subset\Lambda$, we consider the Serre subcategory
$\mathcal{O}^K$ of $\mathcal{O}_\Lambda$ generated by~$\{L(\mu)\,|\,\mu\in K\}$.
The projective cover of $L(\mu)$ in $\mathcal{O}^K$ will be denoted by~$P^K(\mu)$
and is, by construction, the largest quotient of $P(\mu)$ which belongs to~$\mathcal{O}^K$.
If $K$ is saturated, then $\Delta(\mu)\in \mathcal{O}^K$ if and only if~$\mu\in K$.

\subsection{Projective functors}\label{s2.4}

A {\em projective functor} on $\mathcal{O}$ is a  direct summand of a functor of the form
$-\otimes V$, where $V$ is a finite dimensional $\mathfrak{g}$-module. By \cite[Theorem~3.3]{BG},
isomorphism classes of indecomposable projective functors on $\mathcal{O}_0$ are in bijection
with elements in $W$. For $w\in W$, we denote by $\theta_w$ the unique (up to isomorphism)
indecomposable projective functor on $\mathcal{O}_0$ such that
\begin{equation}\label{DefTheta}
\theta_w \Delta(e)\cong P(w).
\end{equation}
For any $w\in W$, the pair $(\theta_w,\theta_{w^{-1}})$ is an adjoint pair of functors,
see \cite[Section~3]{BG}.

By \cite[4.12]{Jantzen}, for any $w\in W$ and $s\in S$ such that $\ell(ws)>\ell(w)$,
we have an isomorphism and a short exact sequence as follows:
\begin{equation}\label{eqtheDel}
\theta_s\Delta(w)\cong \theta_s\Delta(ws),\quad\mbox{ and }
\qquad 0\to \Delta(w)\to \theta_s\Delta(w)\to \Delta(ws)\to 0.
\end{equation}

For $x,y\in W$, we have 
\begin{equation}\label{neweq123}
\theta_x L(y)\neq 0 \quad\text{ if and only if }\quad x^{-1}\leq_{\mathcal{L}}y
\quad\text{ if and only if }\quad x\leq_{\mathcal{R}}y^{-1},
\end{equation}
see \cite[Equation~(1)]{KiM}.

\subsection{Twisting functors}\label{s2.5}

For every simple reflection~$s$ in $W$, we have the corresponding right exact {\em twisting functor} $T_s$ on~$\mathcal{O}_\Lambda$, see e.g. \cite{Arkhipov,AS,Kho}.
By \cite[Lemma~2.1(5)]{AS}, we have
\begin{equation}\label{commProj}T_s \,\theta\;\cong \;\theta\,T_s,\end{equation} for any projective functor
$\theta$.  For any $\mu\in\Lambda$ with $s\cdot\mu\le \mu$, we have
\begin{equation}\label{twistVerma}
T_s\Delta(\mu)\;\cong\;\Delta(s\cdot\mu),
\end{equation}
see e.g. \cite[Lemma~5.7]{CM1}. In general, for any $\mu\in\Lambda$, we have
\begin{equation}\label{twistVerma2}
[T_s\Delta(\mu)]=[\Delta(s\cdot\mu)]
\end{equation}
in the Grothendieck group of $\mathcal{O}$, see \cite[Lemma~2.1(3)]{AS}.

By \cite[Theorem~2]{Kho}, twisting functors satisfy braid relations, which allows us to
unambiguously define
\begin{equation}\label{eq2}
T_w:= T_{s_1}T_{s_2}\cdots T_{s_k},
\end{equation}
where $w=s_1\cdots s_k$ is a reduced expression.
By \cite[Theorem~2.2 and Corollary~4.2]{AS}, $T_w$ induces an isomorphism
\begin{equation}\label{isomPT}
T_w:\;{\mathrm{Hom}}_{\mathcal{O}}(M,N)\;\tilde\to\; {\mathrm{Hom}}_{\mathcal{O}}(T_wM,T_w N)
\end{equation}
for any two modules $M,N$ with $\Delta$-flag. More generally, the left derived functor of
$T_w$ is an autoequivalence of $\mathcal{D}^b(\mathcal{O})$.

The functor  $T_s$, where $s\in S$, is defined in \cite{Arkhipov,AS} via tensoring with a certain
semi-infinite $U$-$U$-bimodule. This means that $T_w$ is defined as a functor on the category of
all $U$-modules. Furthermore, the definition is applicable to any Lie algebra with a fixed
$\mathfrak{sl}_2$-subalgebra. From the definition in \cite[Sections~2.1 and 2.3]{Arkhipov},
we have that
\begin{equation}\label{TwistInd}
T_s\circ{\rm Ind}^{\mathfrak{g}}_{\fk}\;\cong\;{\rm Ind}^{\mathfrak{g}}_{\fk}\circ T_s,\end{equation}
for any subalgebra $\fk\subset \mathfrak{g}$ which contains the $\mathfrak{sl}_2$-subalgebra corresponding
to the simple reflection~$s$.

Let $G_s$, where $s\in S$, denote {\em Joseph's completion functor} which is right adjoint to $T_s$, see
\cite{Kho}. For a reduced expression $w=s_1s_2\cdots s_k$, we define
\begin{displaymath}
G_w:= G_{s_k}\cdots G_{s_2}G_{s_1}
\end{displaymath}
and have that $(T_w,G_w)$ is an adjoint pair.

\subsection{Graded versions}\label{s2.6}

The principal origin of the grading on $\mathcal{O}$ comes from the center of $\mathcal{O}$
which is isomorphic to the coinvariant algebra
\begin{displaymath}
\mathtt{C}:= S(\mathfrak{h})/\langle S(\mathfrak{h})_+^{W}\rangle 
\end{displaymath}
naturally graded by setting the degree of $\mathfrak{h}$ to be $2$.

We denote by $\mathcal{O}^{\mathbb{Z}}$ the {\em $\mathbb{Z}$-graded version} of $\mathcal{O}$ where
in each block of $\mathcal{O}$ we fix the corresponding Koszul grading, see \cite{BGS}.
We denote by $\langle 1\rangle$  the functor which decreases the grading by $1$.

All structural modules (simples, projectives, Vermas etc.) admit graded lifts. We fix the {\em standard
graded lift}, which we will denote by the same symbol as the corresponding ungraded module, as follows:
\begin{itemize}
\item for simple modules, their standard graded lifts are concentrated in degree zero;
\item for Verma modules, their standard graded lifts are such that the natural projection onto the
simple top is homogeneous of degree zero;
\item for projectives modules, their standard graded lifts are such that the natural projection onto the
simple top is homogeneous of degree zero.
\end{itemize}

Both projective and twisting functors admit graded lifts. For $\theta_w$, where $w\in W$, we fix its
{\em standard graded lift}, which we will denote by the same symbol, such that \eqref{DefTheta} holds in
$\mathcal{O}^{\mathbb{Z}}$. For $T_s$, where $s\in S$, we fix its
{\em standard graded lift}, which we will denote by the same symbol, such that \eqref{twistVerma} holds in
$\mathcal{O}^{\mathbb{Z}}$. By \eqref{eq2}, this uniquely defines standard graded lifts of all
$T_w$, where $w\in W$. We refer the reader to \cite{pairing,SHPO2,KiM}  for details.

\subsection{Shuffling functors}\label{SecShuff}

The shuffling functor $C_s$ corresponding to a simple reflection $s$, 
see~\cite{Ca, simple}, is the endofunctor of $\mathcal{O}_0$ defined as the cokernel of the 
adjunction morphism from the identity functor to the projective functor $\theta_s$. 
The graded lift of $C_s$ is defined by the exact sequence
\begin{displaymath}
\Id\langle -1\rangle \to \theta_s\to C_s\to 0. 
\end{displaymath}

For any $w\in W$ with reduced expression $w=s_{1}s_2\cdots s_m$, we can 
define the functor
\begin{displaymath}
C_{w}=C_{s_m}C_{s_{m-1}}\cdots C_{s_1}, 
\end{displaymath}
where the resulting functor does not depend on the choice of a reduced expression, 
see \cite[Lemma~5.10]{shuffling} or \cite[Theorem~2]{Kho} and 
\cite[Section~6.5]{MOS}. The functor $C_w$ is right exact and the corresponding left derived
functor $\cL C_w$ is an auto-equivalence of $\mathcal{D}^b(\mathcal{O}_0)$, see 
\cite[Theorem~5.7]{shuffling}. The cohomology functors of $C_w$ vanish on modules with Verma flag, see \cite[Proposition~5.3]{shuffling}.

\section{Some special cases and general reductions}\label{s3}

\subsection{Koszul-Ringel duality reduction}\label{s3.1}

For $M\in\mathcal{O}$, denote by $\mathbf{n}(M)$ the {\em number of indecomposable summands} of $M$.
Then Conjecture~\ref{conj1} can be reformulated as the conjecture that, in type A, the function
\begin{displaymath}
f:W\times W\to \{0,1,2,\dots\}\quad \text{ defined via }f(x,y):=\mathbf{n}(\theta_x L(y)),
\end{displaymath}
has only values $0$ and $1$.

By \cite{SoergelD}, the category $\mathcal{O}_0$ is Koszul self-dual. By
\cite{SoergelT}, the category $\mathcal{O}_0$ is Ringel self-dual.
By \cite[Theorem~16]{SHPO2}, the composition of Koszul and Ringel dualities on
$\mathcal{D}^b(\mathcal{O}_0)$ maps $\theta_x L(y)$
to $\theta_{y^{-1}w_0} L(w_0x^{-1})$. Both of these dualities are equivalences of 
certain derived categories which are Krull-Schmidt. Therefore, for all $x,y\in W$, we have
\begin{equation}\label{eq7}
f(x,y)=f(y^{-1}w_0,w_0x^{-1}).
\end{equation}

\subsection{Invariance under Kazhdan-Lusztig cells}\label{s3.2}

Let $\leq_{\mathcal{L}}$, $\leq_{\mathcal{R}}$ and $\leq_{\mathcal{J}}$ denote the
{\em Kazhdan-Lusztig left, right and two-sided} pre-orders on $W$, respectively.
Let $\sim_{\mathcal{L}}$, $\sim_{\mathcal{R}}$ and $\sim_{\mathcal{J}}$ denote the
corresponding equivalence classes (also known as {\em cells}).

\begin{proposition}\label{prop21}
Assume that we are in type A. Let $x,x',y,y'\in W$ be such that
$x\sim_{\mathcal{R}}x'$ and $y\sim_{\mathcal{L}}y'$. Then $f(x,y)=f(x',y')$.
\end{proposition}

\begin{proof}
That  $f(x',y)=f(x',y')$, is proved in the proof of \cite[Proposition~11]{KiM}.
Applying  \eqref{eq7} to this statement (for all $x',y$ and $y'$), we get  
$f(x,y)=f(x',y)$ and the claim follows.
\end{proof}

Let $I(W)$ denote the set of all involutions in $W$. Let $I'(W)$ denote the set of
all elements in $W$ of the form $ww_0$, where $w\in I(W)$.

\begin{corollary}\label{cor22}
Assume that $\mathfrak{g}$ is of type A. The following claims are equivalent.
\begin{enumerate}[$($a$)$]
\item\label{cor22.1} Conjecture~\ref{conj1} is true.
\item\label{cor22.2} The assertion of Conjecture~\ref{conj1} is true for all $x,y\in I(W)$.
\item\label{cor22.3} The assertion of Conjecture~\ref{conj1} is true for all $x\in I(W)$ and $y\in I'(W)$.
\end{enumerate}
\end{corollary}

\begin{proof}
In type A, any left and any right cell contains a unique involution (the Duflo involution).
If $\mathcal{L}$ is a left (right) cell, then $\mathcal{L}w_0$ is a left (right)  cell as well, see
\cite[Page~179]{BB}. Therefore any left and any right cell contains a unique element
in $I'(W)$. Hence equivalence of \eqref{cor22.1}--\eqref{cor22.3} follows directly from
Proposition~\ref{prop21}.
\end{proof}

\subsection{Special cases}\label{s3.3}

In \cite[Section~5]{KiM}, the following special cases (here $x,y\in W$) are given:
\begin{itemize}
\item In type A, from $x\sim_{\mathcal{J}}y$ it follows that $f(x,y)\in\{0,1\}$.
\item In any type, from $x=w_0^{\mathfrak{p}}$, for some $\mathfrak{p}$, it follows that $f(x,y)\in\{0,1\}$.
\item In any type, from $y=w_0^{\mathfrak{p}}w_0$, for some $\mathfrak{p}$, it follows that we have $f(x,y)\in\{0,1\}$
(combine the previous case with \eqref{eq7}).
\item In any type, from $y\in I(W)$ and $y\sim_{\mathcal{R}}w_0^{\mathfrak{p}}w_0$,
for some $\mathfrak{p}$, it follows that $f(x,y)\in\{0,1\}$.
\end{itemize}

\subsection{Regular vs singular blocks}\label{s3.4}

One can also ask the question of whether $\theta L(\lambda)$ is indecomposable or zero for any projective
functor $\theta$ and any simple highest weight module $L(\lambda)$. Conjecture~\ref{conj1} formally refers
only to regular $\lambda$. The general case, however, follows easily from the regular one.

\begin{lemma}\label{regvssing}
If  Conjecture~\ref{conj1} is true, then $\theta L(\lambda)$ is indecomposable or zero 
for any projective functor $\theta$ and any simple highest weight module $L(\lambda)$.
\end{lemma}

\begin{proof}
If $\lambda$ is singular, we may write $L(\lambda)$ as the translation to the wall of a regular simple
highest weight modules, say $L(\lambda)=\theta^{\mathrm{on}}L(\mu)$. If we assume  $\theta L(\lambda)\cong M\oplus N$
with non-zero $M$ and $N$, then, translating $M$ and $N$ out of the wall, we get a non-trivial decomposition
$\theta^{\mathrm{out}}M\oplus \theta^{\mathrm{out}}N$. This means that 
$\theta^{\mathrm{out}}\theta\theta^{\mathrm{on}}L(\mu)$ decomposes. As the projective functor 
$\theta^{\mathrm{out}}\theta\theta^{\mathrm{on}}$ is indecomposable (as it sends the dominant Verma module
to an indecomposable projective module, which is the translation out of the wall of an indecomposable
projective module in our singular block), the claim follows. 
\end{proof}

\section{Homological approach}\label{s4}

To simplify notation, in this section
we write $\mathrm{Hom}$ and $\mathrm{Ext}$ instead of  $\mathrm{Hom}_{\mathcal{O}}$ and
$\mathrm{Ext}_{\mathcal{O}}$, respectively.

\subsection{Via twisted projective modules}\label{s4.1}

Recall that a module is indecomposable if and only if its endomorphism algebra is local. Therefore, to prove
that some module is indecomposable it is enough to show that the endomorphism algebra of this module
is either a unital subalgebra or a quotient algebra of some local algebra.

For $w\in W$, we will denote by $K(w)$ the kernel of the natural projection $\Delta(w)\tto L(w)$.
Note that $K(w)$ is naturally graded and lives in strictly positive degrees.

\begin{proposition}\label{prop23}
Assume that we are in any type and let $x,y\in W$.
Assume that the graded vector space $\mathrm{Ext}^1(\theta_x\Delta(y),\theta_x K(y))$
is either zero or lives in strictly positive degrees. Then $f(x,y)\in\{0,1\}$.
\end{proposition}

\begin{proof}
Applying the bifunctor $\mathrm{Hom}({}_-,{}_-)$ to the short exact sequence
\begin{displaymath}
\theta_x K(y)\hookrightarrow \theta_x\Delta(y)\tto  \theta_x L(y),
\end{displaymath}
we get the following diagram in which the bottom row is exact:
\begin{displaymath}
\resizebox{12.5cm}{!}{
\xymatrix{
&\mathrm{Hom}(\theta_x L(y),\theta_x L(y))\ar@{^{(}->}[d]&\\
\mathrm{Hom}(\theta_x\Delta(y),\theta_x\Delta(y))\ar[r]&
\mathrm{Hom}(\theta_x\Delta(y),\theta_x L(y))\ar[r]&
\mathrm{Ext}^1(\theta_x\Delta(y),\theta_x K(y))
}
}
\end{displaymath}
Using Subsection~\ref{s2.5}, we have:
\begin{displaymath}
\begin{array}{rcl}
\mathrm{End}(\theta_x\Delta(y))&\cong&
\mathrm{End}(\theta_xT_y \Delta(e))\\
&\cong&
\mathrm{End}(T_y\theta_x \Delta(e))\\
&\cong&
\mathrm{End}(T_y P(x))\\
&\cong&
\mathrm{End}(P(x)).
\end{array}
\end{displaymath}
As $P(x)$ is an indecomposable projective module over a Koszul algebra,
$\mathrm{End}(P(x))$ is positively graded and local. In particular, it
lives in non-negative degrees and the degree zero component is one-dimensional.

If $\mathrm{Ext}^1(\theta_x\Delta(y),\theta_x K(y))=0$ or if
$\mathrm{Ext}^1(\theta_x\Delta(y),\theta_x K(y))$ lives in strictly positive
degrees, we obtain that $\mathrm{End}(\theta_x L(y))$ is finite dimensional,
non-negatively graded and  the degree zero component is one-dimensional.
Therefore $\mathrm{End}(\theta_x L(y))$ is local and the claim follows.
\end{proof}

The following observations should be compared with the results of Subsection~\ref{s3.3}.

\begin{lemma}\label{lem24}
Assume that we are in any type and let $x,y\in W$ be such that
$x=w_0^{\mathfrak{p}}$, for some $\mathfrak{p}$.
Then we have
\begin{displaymath}
\mathrm{Ext}^1(\theta_x\Delta(y),\theta_x K(y))=0=\mathrm{Ext}^1(\theta_x\Delta(y),K(y)).
\end{displaymath}
\end{lemma}

\begin{proof}
The functor $\theta_x=\theta_{w_0^{\mathfrak{p}}}$ is the translation functor through the
intersection of all walls which correspond to simple reflections of $\mathfrak{l}$.
This functor can be factorized $\theta_x=\theta^{out}\theta^{on}$, where
$\theta^{on}$ is the translation to this intersection and $\theta^{out}$ is the translation
out of this intersection. The two latter functors are biadjoint and the composition
$\theta^{on}\theta^{out}$ is isomorphic to a direct sum of $|W^{\mathfrak{p}}|$ copies of
the identity functor for the singular block on the intersection of these walls. Therefore, using
biadjunction and additivity, it is enough to prove that
$\mathrm{Ext}^1(\theta_x\Delta(y),K(y))=0$. Again, using biadjunction, we have
\begin{displaymath}
\mathrm{Ext}^1(\theta_x\Delta(y),K(y))=
\mathrm{Ext}^1(\theta^{on} \Delta(y),\theta^{on} K(y)).
\end{displaymath}

Now, $\theta^{on} \Delta(y)$ is a Verma module in the singular block.
If $\theta_x L(y)=0$, then $\theta_x\Delta(y)=\theta_x K(y)$ and we get the assertion of the
lemma due to the fact that Verma modules have no self-extensions in $\mathcal{O}$.
If $\theta_x L(y)\neq 0$, we note that each composition subquotient $L(z)$
of $K(y)$ satisfies $z\succ y$ and this is preserved by translation to the singular block.
therefore the assertion of the lemma follows from the fact that the singular block
is a highest weight category and hence $\theta^{on} \Delta(y)$ is relatively projective
with respect to $\theta^{on} K(y)$ because of the ordering property above.
\end{proof}

\begin{lemma}\label{lem24b}
Assume that we are in any type and let $x,y\in W$ be such that
$y=w_0^{\mathfrak{p}}w_0$, for some $\mathfrak{p}$.
Then $\mathrm{Ext}^1(\theta_x\Delta(y), K(y))=0$.
\end{lemma}

\begin{proof}
Under our assumptions on $y$, we may consider the BGG resolution of $K(y)$, 
that is a complex
\begin{displaymath}
\mathcal{X}_{\bullet}:\qquad 0\to X_{-k}\to \dots \to X_{-1} \to X_0\to 0
\end{displaymath}
obtained, using the parabolic induction and then truncation, 
from the classical BGG-resolution, see  \cite{BGGr}. Here
\begin{displaymath}
X_0=\bigoplus_{w\in W^{\mathfrak{p}}:\ell(w)=1}\Delta(wy),\quad 
X_{-1}=\bigoplus_{w\in W^{\mathfrak{p}}:\ell(w)=2}\Delta(wy),\quad \dots.
\end{displaymath}
Note that each $\Delta(y')$ appearing ion this resolution satisfies
$y'\succ y$ and can be written
as $T_y\Delta(y'')$, where $y''\preceq w_0w_0^{\mathfrak{p}}w_0$. Therefore
$\mathcal{X}_{\bullet}$ has the form $T_y \mathcal{Y}_{\bullet}$ for some
complex $\mathcal{Y}_{\bullet}$ of Verma modules concentrated in non-positive positions.
Now, using Subsection~\ref{s2.5}, $\mathrm{Ext}^1(\theta_x\Delta(y), K(y))$ can be rewritten
as
\begin{displaymath}
\mathrm{Hom}_{\mathcal{D}^b(\mathcal{O})}(P(x),\mathcal{Y}_{\bullet}[1])
\end{displaymath}
and the latter is zero as the shifted complex $\mathcal{Y}_{\bullet}[1]$ has only zero in position zero.
\end{proof}

Lemma~\ref{lem24b} implies, by adjunction, that
$\mathrm{Ext}^1(\theta_x\Delta(y), \theta_{x'}K(y))=0$, for all $x'\in W$ and all $x,y$ as in the formulation
of the lemma.

The above observations naturally raise the following questions:

\begin{question}\label{quest25}
For which $x,y\in W$ does the space $\mathrm{Ext}^1(\theta_x\Delta(y),\theta_x K(y))$ live
in strictly positive degrees?
\end{question}

\begin{question}\label{quest26}
For which $x,y\in W$ do we have $\mathrm{Ext}^1(\theta_x\Delta(y),\theta_x K(y))=0$?
\end{question}

\begin{question}\label{quest27}
For which $x,y\in W$ do we have $\mathrm{Ext}^1(\theta_x\Delta(y),K(y))=0$?
\end{question}

None of these questions seems to be easy in full generality, for instance,
due to the following example.

\begin{example}\label{ex28}
{\rm
For $\mathfrak{g}=\mathfrak{sl}_3$, we have $W=\{e,s,t,st,ts,w_0\}$. Using the adjunction
$(T_s,G_s)$, we obtain that $\mathrm{Ext}^1(\theta_{ts}\Delta(s),K(s))$ is isomorphic to
$\mathrm{Hom}(P(ts),\mathcal{R}^1G_s K(s))$, where $\mathcal{R}^1G_s$ is the functor of
taking the maximal $s$-finite quotient. In the case of  $K(s)$, this quotient
is isomorphic to $L(ts)$. Tracking the grading, one sees that this $L(ts)$ does live in degree zero.
Hence $\mathrm{Ext}^1(\theta_{ts}\Delta(s),K(s))$ is one-dimensional and is concentrated in
degree zero. At the same time, as $\theta_{ts}L(ts)=\theta_{s}\theta_{t}L(ts)=0$, it follows that
$\mathrm{Ext}^1(\theta_{ts}\Delta(s),\theta_{ts}K(s))=0$.
}
\end{example}

We finish this subsection with the following observation.

\begin{proposition}\label{prop29}
Assume that we are in any type and let $x,y\in W$.
Then the graded vector space $\mathrm{Ext}^1(\theta_x\Delta(y),K(y))$, if non-zero,
lives in non-negative degrees.
\end{proposition}

\begin{proof}
By adjunction, we can rewrite $\mathrm{Ext}^1(\theta_x\Delta(y),K(y))$ as
\begin{displaymath}
\mathrm{Hom}_{\mathcal{D}^b(\mathcal{O})}(P(x),\mathcal{R}G_y K(y)[1])
\end{displaymath}
and hence reduce the statement to that of module $\mathcal{R}^1G_y K(y)$ living in
non-negative degrees. Note that $K(y)$ lives in strictly positive degrees, by construction.

By \cite{AS,Kho,simple}, the functor $\mathcal{R}G_s$ has derived length two with the zero component
$G_s$ and the first homology component $Z_s\langle 1\rangle$, where $Z_s$ is the corresponding
{\em Zuckerman} functor of taking the maximal $s$-finite quotient. Applying $G_s$ to a module
concentrated in degree $0$, either gives $0$ or a module concentrated in degrees $0$ and $1$
(this is the dual of \cite[Section~6]{AS}).
Therefore $G_s$ cannot lower the minimum non-zero degree of a graded module. In turn,
$Z_s\langle 1\rangle$ may lower this minimum degree by one.

Now, writing $\mathcal{R}G_y$ as a composition of various $\mathcal{R}G_s$, where $s$ is a simple reflection,
we can use the previous paragraph and the fact that $K(y)$ lives in strictly positive degrees
to deduce that $\mathcal{R}^1G_y K(y)$ lives in non-negative degrees.
\end{proof}

\subsection{By induction with respect to the Bruhat order}\label{s4.2}
In this subsection we work in type A.

A (much more subtle) variation of the approach described in the previous subsection is
to prove Conjecture~\ref{conj1} by the downward induction on $y$ with respect to the Bruhat order.
The basis of the induction, that is the case $y=w_0$, is covered by Subsection~\ref{s3.3}.
So, the trick is to prove the induction step. Note also that, by Ringel self-duality of $\mathcal{O}_0$,
we also know that each algebra $\mathrm{End}(\theta_x L(w_0))$ is non-negatively graded with one
dimensional  zero component.

Let $y\in W$ and $s\in S$ be such that $sy\preceq y$. Then $\mathrm{Ext}^1(L(sy),L(y))$ is one-dimensional.
Let $M_{y,s}$ be an indecomposable module which fits into a short exact sequence
\begin{equation}\label{eq3}
0\to L(y)\to M_{y,s}\to L(sy)\to 0.
\end{equation}
Due to Proposition~\ref{prop21}, without loss of generality we may assume that $y$ and $sy$ belong
to different left cells and hence also to different two-sided cells.

\begin{theorem}\label{thm31}
Assume that $y$ and $s$ are as above and $x\in W$. Then, under the vanishing condition
$\mathrm{Ext}^1(\theta_x M_{y,s},\theta_x L(y))=0$,
the fact that $\mathrm{End}(\theta_x L(y))$ is non-negatively graded with one dimensional
zero component implies that $\mathrm{End}(\theta_x L(sy))$
is non-negatively graded with one dimensional  zero component.
\end{theorem}

\begin{proof}
Applying the bifunctor $\mathrm{Hom}({}_-,{}_-)$ to the image of
the short exact sequence \eqref{eq3} under $\theta_x$, we obtain
the following commutative diagram:
\begin{displaymath}
\resizebox{12.5cm}{!}{\xymatrix{
\mathrm{Hom}(\theta_xL(sy),\theta_xL(y))\ar@{^{(}->}[d]\ar@{^{(}->}[r]&
\mathrm{Hom}(\theta_xL(sy),\theta_xM_{y,s})\ar@{^{(}->}[d]\ar[r]&
\mathrm{Hom}(\theta_xL(sy),\theta_xL(sy))\ar@{^{(}->}[d]&\\
\mathrm{Hom}(\theta_xM_{y,s},\theta_xL(y))\ar@{^{(}->}[r]\ar[d]&
\mathrm{Hom}(\theta_xM_{y,s},\theta_xM_{y,s})\ar[d]\ar[r]&
\mathrm{Hom}(\theta_xM_{y,s},\theta_xL(sy))\ar[d]\ar[r]&
\mathrm{Ext}^1(\theta_xM_{y,s},\theta_xL(y))\\
\mathrm{Hom}(\theta_xL(y),\theta_xL(y))\ar@{^{(}->}[r]&
\mathrm{Hom}(\theta_xL(y),\theta_xM_{y,s})\ar[r]&
\mathrm{Hom}(\theta_xL(y),\theta_xL(sy))&
}}
\end{displaymath}

By our assumptions, $sy<_{\mathcal{J}}y$.
Recall that all top components and all socular components of $\theta_x L(y)$ belong to the
right cell of $L(y)$ and all composition subquotients of $\theta_x L(y)$ are less than or equal
to $y$ with respect to the right order. Further, each simple in the socle of $\theta_x I(y)$ is greater than
or equal to $y$ with respect to the two-sided  order. It follows that 
$\theta_xL(sy)$ cannot contribute to the socle of $\theta_xM_{y,s}$ and that none of the socular or top
components of $\theta_xL(y)$ appears in $\theta_xL(sy)$. Taking all
this into account, we get
\begin{equation}\label{eq4}
\resizebox{11.5cm}{!}{$
\mathrm{Hom}(\theta_xL(sy),\theta_xL(y))=
\mathrm{Hom}(\theta_xL(sy),\theta_xM_{y,s})=
\mathrm{Hom}(\theta_xL(y),\theta_xL(sy))=0.$
}
\end{equation}
Therefore
\begin{displaymath}
\mathrm{End}(\theta_x L(sy))\cong \mathrm{Hom}(\theta_xM_{y,s},\theta_xL(sy))
\end{displaymath}
and, under the assumption $\mathrm{Ext}^1(\theta_x M_{y,s},\theta_x L(y))=0$,
we get a surjection
\begin{equation}\label{eq5}
\mathrm{End}(\theta_x M_{y,s})\tto\mathrm{End}(\theta_x L(sy)).
\end{equation}

Furthermore, \eqref{eq4} also yields
\begin{displaymath}
\mathrm{End}(\theta_xL(y))\cong
\mathrm{Hom}(\theta_xL(y),\theta_xM_{y,s})
\end{displaymath}
and
\begin{displaymath}
\mathrm{End}(\theta_x M_{y,s})\hookrightarrow\mathrm{Hom}(\theta_xL(y),\theta_xM_{y,s}).
\end{displaymath}
Therefore $\mathrm{Hom}(\theta_xL(y),\theta_xM_{y,s})$ is non-negatively graded with
one dimensional zero component by the inductive assumption, which implies that
$\mathrm{End}(\theta_x M_{y,s})$ is non-negatively graded with
one dimensional zero component and, via \eqref{eq5}, that $\mathrm{End}(\theta_x L(sy))$
 is non-negatively graded with one dimensional  zero component.
\end{proof}

As an immediate consequence from Theorem~\ref{thm31} we have:

\begin{corollary}\label{cor32}
Assume that  $\mathrm{Ext}^1(\theta_x M_{y,s},\theta_x L(y))=0$, for all $x,y$ and $s$
as in Theorem~\ref{thm31}. Then Conjecture~\ref{conj1} is true.\hfill$\qed$
\end{corollary}

This motivates the following questions:

\begin{question}\label{quest35}
For which $x,y\in W$ and $s\in S$ does $\mathrm{Ext}^1(\theta_x M_{y,s},\theta_x L(y))$ vanish?
\end{question}

\begin{question}\label{quest36}
For which $x,y\in W$ and $s\in S$ does $\mathrm{Ext}^1(\theta_x M_{y,s},L(y))$ vanish?
\end{question}

These questions, in turn, motivate the following related, but not equivalent questions:

\begin{question}\label{quest37}
For which $x,y\in W$ do we have $\mathrm{Ext}^1(\theta_x L(y),\theta_x L(y))=0$?
\end{question}

\begin{question}\label{quest38}
For which $x,y\in W$ do we have $\mathrm{Ext}^1(\theta_x L(y),L(y))=0$?
\end{question}

Again, all the above questions can be answered in the special
case $x=w_0^{\mathfrak{p}}$, for some $\mathfrak{p}$.

\begin{lemma}\label{lem39}
If $x=w_0^{\mathfrak{p}}$, for some $\mathfrak{p}$, then
$\mathrm{Ext}^1(\theta_x L(y),L(y))=0$, for all $y\in W$. In particular,
$\mathrm{Ext}^1(\theta_x L(y),\theta_x L(y))=0$.
\end{lemma}

\begin{proof}
Similarly to the proof of Lemma~\ref{lem24}, it is enough to prove the first vanishing property
$\mathrm{Ext}^1(\theta_x L(y),L(y))=0$. By adjunction, this is equivalent to
$\mathrm{Ext}^1(\theta^{on} L(y),\theta^{on} L(y))=0$. If $\theta^{on} L(y)=0$, we are done.
In case we have $\theta^{on} L(y)\neq 0$, the module $\theta^{on} L(y)$ is a simple module in the singular block.
As category $\mathcal{O}$ has finite projective dimension, simple modules cannot have
non-zero first extensions. The claim follows.
\end{proof}

\begin{lemma}\label{lem391}
If $x=w_0^{\mathfrak{p}}$, for some $\mathfrak{p}$, then
$\mathrm{Ext}^1(\theta_x M_{y,s},L(y))=0$, for all $y\in W$ and $s\in S$ as above. In particular,
$\mathrm{Ext}^1(\theta_x M_{y,s},\theta_x L(y))=0$.
\end{lemma}

\begin{proof}
As above, it is enough to prove $\mathrm{Ext}^1(\theta_x M_{y,s},L(y))=0$ which is equivalent to
$\mathrm{Ext}^1(\theta^{on} M_{y,s},\theta^{on} L(y))=0$. If $\theta^{on} L(y)=0$, the claim is clear.
If $\theta^{on} L(sy)=0$, the claim reduces to Lemma~\ref{lem39} due to exactness of $\theta^{on}$.
If $\theta^{on} M_{y,s}$ has length two, then it is an indecomposable module and the claim follows
from the fact that $\mathrm{Ext}^1(\theta^{on} L(sy),\theta^{on} L(y))$ is one-dimensional.
\end{proof}

\section{Higher representation theoretic approach}\label{s7}

\subsection{Basics on $2$-categories and $2$-representations}\label{s7.0}

In this subsection we recall some basics on finitary $2$-categories and their $2$-representations 
following \cite{MM1,MM3,MM5,MM6}. For generalities on $2$-categories, we refer the reader to \cite{McL}.

A {\em $2$-category} is a category enriched over the category of small categories. A {\em finitary}
$2$-category is a $2$-category $\cC$ for which all relevant structural information is finite.
In particular, such $\cC$ has finitely many objects, its morphisms categories are additive,
$\mathbb{C}$-linear Krull-Schmidt categories with finitely many isomorphism classes of 
indecomposable objects and finite dimensional morphism spaces. All identity $1$-morphisms in
$\cC$ are indecomposable.

Functors between $2$-categories preserving all relevant $2$-structure are called {\em $2$-functors}.
A {\em $2$-representation} of a $2$-category is a $2$-functor to a suitable $2$-category.
Quite often, a $2$-representation of a $2$-category $\cC$ can be viewed as a functorial action
of $\cC$ on some category. All $2$-representations of $\cC$ form a $2$-category where
$1$-morphisms are {\em $2$-natural transformations} and $2$-morphisms are {\em modifications}.
For finitary $2$-categories, the most natural $2$-representations are functorial actions, by additive functors,
on additive $\mathbb{C}$-linear Krull-Schmidt categories. 

A finitary $2$-category $\cC$ is called {\em fiat} if it has a weak involution and the corresponding
adjunction morphisms. 

For two indecomposable $1$-morphisms $\mathrm{F}$ and $\mathrm{G}$ of a finitary $2$-category $\cC$,
one writes $\mathrm{F}\geq_L\mathrm{G}$ provided that $\mathrm{F}$ is isomorphic to a direct summand 
of $\mathrm{H}\circ \mathrm{G}$, for some $1$-morphism $\mathrm{H}$. The defines the {\em left preorder}
$\geq_L$ and the corresponding equivalence classes are called {\em left cells}. 
The   {\em right preorder} $\geq_L$ and the corresponding {\em right cells}, and also the
{\em two-sided preorder} $\geq_J$ and the corresponding {\em two-sided cells} cells are define
similarly using multiplication from the right, or from both sides, respectively.

\subsection{Ingredients}\label{s7.1}

We denote by $\cP$ the fiat $2$-category of projective functors acting on $\mathcal{O}_0$, see
\cite[Subsection~7.1]{MM1} for details. Up to isomorphism, indecomposable $1$-morphisms in $\cP$
are exactly $\theta_w$, $w\in W$. Following the conventions of \cite{MM5}, due to the right
nature of the action of $\cP$ on $\mathcal{O}_0$, {\em left cells} of $\cP$ are indexed by  
right Kazhdan-Lusztig cells in $W$, {\em right cells} of $\cP$ are indexed by  
left Kazhdan-Lusztig cells in $W$ and {\em two-sided cells} of $\cP$ are indexed by  
two-sided Kazhdan-Lusztig cells in $W$. The notion of {\em Duflo involution} for $\cP$,
see \cite[Section~4]{MM1} corresponds to the notion of Duflo involution in $W$.

If $\mathfrak{g}$ is of type $A$, all two-sided cells in $\cP$ are {\em strongly regular} in the
sense that the intersection of any left and any right cell inside the same two-sided cell 
consists of exactly one indecomposable $1$-morphism. 

\subsection{$2$-subcategories of $\cP$ generated by involutions}\label{s7.2}

For $x\in I(W)$, we denote by $\cP_x$ the $2$-full $2$-subcategory of $\cP$ whose $1$-morphisms are
all endofunctors  of $\mathcal{O}_0$ which belong to the additive closure of all endofunctors of the
form 
\begin{displaymath}
\theta_x^k:=
\begin{cases}
\underbrace{\theta_x\circ\theta_x\circ\dots\circ\theta_x}_{k \text{ factors}},&
k\geq 1;\\ 
\theta_e,& k=0.
\end{cases}
\end{displaymath}

\begin{lemma}\label{lem61-1}
Let $A$ be a finite dimensional, basic, non-simple, associative, unital algebra 
over an algebraically closed field $\Bbbk$ and $\{e_1,e_2,\dots,e_m\}$ 
a complete set of pairwise orthogonal primitive idempotents in $A$. 
Then, for any $i,j,k\in\{1,2,\dots,m\}$,
we have
\begin{displaymath}
\big(Ae_i\otimes_{\Bbbk}e_jA\big)\otimes_A \big(Ae_j\otimes_{\Bbbk}e_kA\big)\cong
\big(Ae_i\otimes_{\Bbbk}e_kA\big)^{\oplus \dim(e_j Ae_j)}\neq 0.
\end{displaymath}
\end{lemma}

\begin{proof}
This is a direct computation which follows directly from the definitions. 
\end{proof}

We recall that, by \cite[Theorem~13]{MM3}, given a fiat $2$-category $\cC$ and a strongly regular
two-sided cell $\mathcal{J}$ in $\cC$, there is a unique maximal $2$-ideal $\cI$ in $\cC$ such that
$\mathcal{J}$ becomes a unique maximal two-sided cell in the $\mathcal{J}$-simple quotient
$\cC/\cI$. Furthermore, there is a finite dimensional algebra $A$ such that $1$-morphisms in 
$\mathcal{J}$ in the quotient $\cC/\cI$ can be viewed as indecomposable
projective functors for a finite dimensional associative algebra. With the notation as in
Lemma~\ref{lem61-1}, we have that  left cells inside $\mathcal{J}$ have the form
$\{Ae_i\otimes_{\Bbbk}e_jA:i=1,\dots,m\}$, where $j$ is arbitrary, right cells inside $\mathcal{J}$ have the form
$\{Ae_i\otimes_{\Bbbk}e_jA:j=1,\dots,m\}$,  where $i$ is arbitrary, and each left or right cell contains a unique
{\em Duflo involution}, namely a $1$-morphism of the form $Ae_i\otimes_{\Bbbk}e_iA$.
Finally, the adjoint of $Ae_i\otimes_{\Bbbk}e_jA$ is given by $Ae_j\otimes_{\Bbbk}e_iA$, for all $i$ and $j$.
We refer to \cite[Section~5]{MM5} for details.

\begin{proposition}\label{prop61}
The  $2$-category  $\cP_x$ is fiat and all of its two-sided cells are strongly regular. 
\end{proposition}

\begin{proof}
As $x$ is an involution, $\theta_x$ is self-adjoint, and hence all $\theta_x^k$ are self-adjoint as well.
This implies that $\cP_x$ is stable under taking adjoint functors and hence, being a $2$-full subcategory
of a fiat $2$-category, is fiat.
 
Let $\mathcal{J}$ be a two-sided cell in $\cP$. Assume that this two-sided cell contains
some $\theta_z$ which also belongs to $\cP_x$. Then $\theta_{z^{-1}}$ is in $\cP_x$ 
as well by the previous paragraph. Considering $\theta_z\theta_{z^{-1}}$ and $\theta_{z^{-1}}\theta_z$,
from  Lemma~\ref{lem61-1} we obtain that $\cP_x$ contains both the Duflo involution 
in the left cell of $\theta_z$
and the Duflo involution in the right cell of $\theta_z$. 

Let $\Gamma$ be a directed graph (without multiple edges in the same direction) 
whose vertices are those Duflo involutions in $\mathcal{J}$
which belong to $\cP_x$. For two such Duflo involutions $\theta_1$ and $\theta_2$,
we say that $\Gamma$ contains a directed edge from $\theta_1$ to $\theta_2$ if and only if the
unique elements in the intersection of the right cell of $\theta_1$ and the left cell of
$\theta_2$ is in $\cP_x$. In particular, $\Gamma$ contains loops in all vertices.
So, the binary relation represented by $\Gamma$ is reflexive. From the previous paragraph 
we also see that this relation is symmetric. 
From Lemma~\ref{lem61-1} it follows that the relation represented by $\Gamma$
is transitive. Therefore it is an equivalence relation. Given an equivalence class
with respect to this relation, we see that $\cP_x$ contains $1$-morphisms in the
intersection between the left cell of any Duflo involution in this equivalence class
and the right cell of any Duflo involution in this equivalence class. This equivalence
class thus constitutes a two-sided cell in $\cP_x$ and this two-sided cell is
strongly regular (as the corresponding two-sided cell of $\cP$ is strongly regular).

This shows that any two-sided cell of $\cP_x$ is strongly regular. The claim follows.
\end{proof}

The example below shows how the case of $x=w_0^{\mathfrak{p}}$, for some $\mathfrak{p}$,
distinguishes itself with respect to $\cP_x$ among other involutions.

\begin{example}\label{ex62}
{\rm
If $x=w_0^{\mathfrak{p}}$, for some $\mathfrak{p}$, then 
$\theta_x\circ\theta_x\cong\theta_x^{\oplus |W_{\mathfrak{p}}|}$
and hence $\theta_x$ is the only indecomposable $1$-morphism in $\cP_x$, up to isomorphism,
which is not isomorphic to the identity.
}
\end{example}

Example~\ref{ex62} should be compared with the following statement:

\begin{lemma}\label{lem63}
Let $x\in W$ and $s_1,s_2,\dots, s_k$ be the list of simple reflections which 
appear in a reduced decomposition of $x$. Let $W'$ be the parabolic subgroup of $W$
generated by these reflections and $w'_0$ be the longest element in $W'$. Then
the  $2$-category  $\cP_x$ contains $\theta_{w'_0}$. 
\end{lemma}

\begin{proof}
Without loss of generality we may assume $S=\{s_1,s_2,\dots, s_k\}$, which means that
$W=W'$ and hence $w_0=w'_0$. It is enough to show that, for any $w\in W$, the Verma flag of
$\theta_x\Delta(w)$ contains all $\Delta(ws)$, where $s\in S$ and $ws\succ w$. Indeed, then, 
by induction,  we will have that $\theta_x^{\ell(w_0)}\Delta(e)$ contains $\Delta(w_0)$ and hence 
$\theta_{w_0}$ is a direct summand of $\theta_x^{\ell(w_0)}$.
 
That $\theta_x\Delta(e)$ contains all  $\Delta(s)$, where $s\in S$, follows directly from 
our assumptions as they guarantee that $x\succ s$, for all $s\in S$. For an arbitrary $w$
the necessary claim is now obtained by applying $T_w$.
\end{proof}

The above raises the following natural question:

\begin{question}\label{qest64}
Which $\theta_z$, where $z\in W$, appear in $\cP_x$, for a fixed $x\in I(W)$?
\end{question}

From \cite[Section~5]{MM1} we know that, for any $x\in I(W)$, we have
\begin{displaymath}
\theta_x\circ\theta_x=\theta_x^{\oplus k}\oplus \theta,
\end{displaymath}
where $k=\dim\mathrm{End}(\theta_x L(x))>0$ and each indecomposable 
direct summand of $\theta$ is strictly $\mathcal{J}$-bigger
that $\theta_x$. Further, if $\theta_z$ appears in $\cP_x$, then, clearly 
$\mathrm{supp}(z)\subset \mathrm{supp}(x)$. As any element different from $w'_0$
and contained in the same $\mathcal{J}$-cell as $w'_0$ must have strictly bigger support,
it follows that all $\theta_z$ appearing in $\cP_x$ and different from $\theta_x$
and $\theta_{w'_0}$ have the property $x<_{\mathcal{J}} z<_{\mathcal{J}} w'_0$.

\subsection{The case $x=w_0^{\mathfrak{p}}$}\label{s7.3}

This subsection provides higher representation theoretic proof for various,
mostly already know, results. Thus the value of this subsection is in
the corresponding methods and arguments, rather than the results. We hope 
that some of these ideas and methods could be extended to 
attack more general cases of Conjecture~\ref{conj1}.

\begin{proposition}\label{prop65}
Let $A$ be a finite dimensional, connected, associative and unital $\Bbbk$-algebra 
and $\mathrm{F}$ an indecomposable and self-adjoint
endofunctor of $A\text{-}\mathrm{mod}$ satisfying $\mathrm{F}\circ \mathrm{F}\cong \mathrm{F}^{\oplus k}$,
for some $k\in\{1,2,\dots\}$. Let, further, $L$ be a simple $A$-module such that $\mathrm{F}\, L\neq 0$. 
Then we have the following:
\begin{enumerate}[$($a$)$]
\item\label{prop65.1}  All indecomposable summands of $\mathrm{F}\, L$ are isomorphic.
\item\label{prop65.2}  For every simple $A$-module $L'$ appearing in the top or socle of $\mathrm{F}\, L$,
we have $\mathrm{add}(\mathrm{F}\, L)=\mathrm{add}(\mathrm{F}\, L')$. 
\item\label{prop65.3}  Every simple $A$-module $L'$ appearing in the top or socle of $\mathrm{F}\, L$
is isomorphic to $L$. 
\end{enumerate}
\end{proposition}

\begin{proof}
Let $\mathrm{F}\, L=M_1\oplus M_2\oplus\dots M_m$, where all $M_i$ are indecomposable.
If $\mathrm{F}\cong \mathrm{Id}_{A\text{-}\mathrm{mod}}$, then the claim is clear, so we assume that 
$\mathrm{F}\not\cong \mathrm{Id}_{A\text{-}\mathrm{mod}}$.

Let $\cC$ be the $2$-full $2$-subcategory of the $2$-category of right exact endofunctors of $A$-mod 
whose $1$-morphisms are all endofunctors in $\mathrm{add}(\mathrm{Id}_{A\text{-}\mathrm{mod}}\oplus \mathrm{F})$.
Then $\cC$ is fiat and has two two-sided cells, 
one consisting of $\mathrm{Id}_{A\text{-}\mathrm{mod}}$ and the other one consisting of $\mathrm{F}$. 
By \cite[Theorem~18]{MM5}, the only simple transitive $2$-representations
of $\cC$ are cell $2$-representations.  As every $M_i$ is in the image of $\mathrm{F}$,
this $M_i$ corresponds to the cell $2$-rep\-re\-sen\-ta\-ti\-on of the cell containing $\mathrm{F}$. 
Now, from \cite[Theorem~25]{ChM} or \cite[Theorem~4]{MM6}, it follows that 
$\mathrm{F} M_i\in\mathrm{add}(M_i)$, for every $i$.

Pick some $M_i$ and let $L'$ be a simple in the top of $M_i$. 
As $\mathrm{F}$ is self-adjoint, by adjunction, $\mathrm{F}\, L'\neq 0$, moreover
$L$ appears in the socle of $\mathrm{F}\, L'$
(and hence also in the top by the self-duality of $\mathrm{F}\, L'$). 
Applying $\mathrm{F}$ to $M_i\tto L'$, we get
\begin{equation}\label{eq8}
M_i^{\oplus k}\cong \mathrm{F} M_i\tto \mathrm{F}\, L'\tto L.
\end{equation}
Applying $\mathrm{F}$ to \eqref{eq8} again, we get 
\begin{displaymath}
M_i^{\oplus k^2}\tto \mathrm{F}\, L\cong M_1\oplus M_2\oplus\dots M_m.
\end{displaymath}
In particular, each $M_j$ is a homomorphic image of $M_i^{\oplus k^2}$.
Now we will need:

\begin{lemma}\label{lem65-0}
Let $B$ be a finite dimensional algebra and $X$ an indecomposable $B$-module. Then,
for any positive integer $n$, any surjection $X^{\oplus n}\tto X$ splits.
\end{lemma}

\begin{proof}
Let $D$ denote the local endomorphism algebra of $X$ and $\mathbf{m}$ the
unique maximal ideal in $D$. Then any $\alpha:X^{\oplus n}\to X$ is given by
a $1\times n$ matrix with coefficients in $D$. If one of those coefficients
is not in $\mathbf{m}$, then this coefficient is an isomorphism and hence
$\alpha$ is a surjection and is, obviously, split. In particular,
if $\mathbf{m}=0$, then the claim is clear. So, in what follows we
assume $\mathbf{m}\neq 0$.

Assume now that $\alpha$ is a surjection and that all coefficients in the
corresponding matrix are in $\mathbf{m}$. As $\mathbf{m}$ is nilpotent, 
it contains a non-zero element $f$ such that $f\mathbf{m}=0$. As $\alpha$
is a surjection, we have $f\circ\alpha\neq 0$. On the other hand,
$fg=0$ for any coefficient $g$ in the matrix of $\alpha$ as $g\in\mathbf{m}$.
This is a contradiction which implies our statement.
\end{proof}

\begin{lemma}\label{lem65-1}
Let $B$ be a finite dimensional algebra and $X$ and $Y$ indecomposable $B$-modules such that  
$X$ is a quotient of some $Y^{\oplus m}$ and $Y$ is a quotient of some $X^{\oplus n}$.
Then $X\cong Y$.  
\end{lemma}

\begin{proof}
By our assumptions, we have $Y^{\oplus mn}\tto X^{\oplus n}\tto Y$. 
By Lemma~\ref{lem65-0}, the epimorphism $Y^{\oplus mn}\tto Y$ splits. Composing this splitting
with the map $Y^{\oplus mn}\tto X^{\oplus n}$, we get a splitting for $X^{\oplus n}\tto Y$. As
$X$ is indecomposable, it follows that $X\cong Y$. 
\end{proof}

As each $M_j$ is a homomorphism image of $M_i^{\oplus k^2}$, from Lemma~\ref{lem65-1}
we obtain that all $M_i$ are isomorphic. This proves claim~\eqref{prop65.1}.

To prove claim~\eqref{prop65.2}, let $L'$ be a simple module such that $\mathrm{F}\,  L\tto L'$
or, equivalently $L'\hookrightarrow \mathrm{F}\,  L$.
By claim~\eqref{prop65.1}, we have $\mathrm{F}\,  L\cong M^{\oplus m}$ and 
$\mathrm{F}\,  L'\cong N^{\oplus m'}$, for some indecomposable $M$ and $N$ and some
positive integers $m$ and $m'$. Applying $\mathrm{F}$ to 
$\mathrm{F}\,  L\tto L'$, we get  $\mathrm{F}^{\oplus k}\, L\cong \mathrm{F}\circ
\mathrm{F}\, L\, \tto \mathrm{F}\,  L'$,
in particular, $M^{\oplus p}\tto N$, for some $p$.

By adjunction, $\mathrm{F}\,  L'\tto L$. Therefore the previous paragraph gives that 
$N^{\oplus q}\tto M$, for some $q$. Therefore $M\cong N$ by Lemma~\ref{lem65-1},
implying claim~\eqref{prop65.2}.

The $2$-category $\cC$ is fiat with strongly regular two-sided cells,
it has two two-sided cells and hence two equivalence classes of 
cell $2$-representations. One of these is annihilated by $\mathrm{F}$, while, in the other
one, we have $\mathrm{F}\, P\cong P^{\oplus k}$, where $P$ is the unique, up to isomorphism, 
indecomposable object. 

Let $I$ be an indexing set of the isomorphism classes of indecomposable $A$-mo\-du\-les. 
Let $\mathtt{Q}:=(m_{i,j})_{i,j\in I}$ be the matrix describing the multiplicity of 
$P(i)$ in $\mathrm{F} P(j)$, for $i,j\in I$. Then, from the previous paragraph and 
\cite[Theorem~25]{ChM}, it follows
that one can order elements in $I$ such that $\mathtt{Q}$ has the form
\begin{displaymath}
\left(\begin{array}{cc}0&0\\ *&kE\end{array}\right),
\end{displaymath}
where $E$ is the identity matrix. By adjunction, the transpose of this matrix gives the matrix counting 
the decomposition multiplicities $[\mathrm{F}\, L(i):L(j)]$.
Consequently, $\mathrm{F} P(i)\cong P(i)^{\oplus k}$ provided that $\mathrm{F}\, L(i)\neq 0$.
This means that $P(i)^{\oplus k} \tto \mathrm{F}\, L(i)$ and proves claim~\eqref{prop65.3}.
\end{proof}

Proposition~\ref{prop65} applies, in particular, to the situation when $A\text{-}\mathrm{mod}\cong \mathcal{O}_0$
and $\mathrm{F}\cong\theta_{w_0^{\mathfrak{p}}}$, for some $\mathfrak{p}$.
The proof of Proposition~\ref{prop65} raises the following natural questions.

\begin{question}\label{quest68}
For which $x\in I(W)$ and $y\in W$, does the module $L(y)$ appear in the top of any summand in $\theta_x L(y)$? 
\end{question}

\begin{question}\label{quest69}
For which $x\in I(W)$ and $y\in W$, none of the summands in $\theta_x L(y)$ is isomorphic to a summand
of $\theta_{x'} L(y)$, for some $x'\in I(W)$ such that $x<_{\mathcal{J}}x'$ and $\theta_{x'}$ 
is a $1$-morphism in $\cP_x$? 
\end{question}

\begin{corollary}\label{cor697}
Assume that $x=w_0^{\mathfrak{p}}$, for some $\mathfrak{p}$. Let  $y\in W$ be such that 
$\theta_x L(y)\neq 0$. Then the module $\theta_x L(y)$ is indecomposable.
\end{corollary}

\begin{proof}
Let us first consider the case $y=x$. As $x=w_0^{\mathfrak{p}}$, the $1$-morphism
$\theta_x$ is the Duflo involution in its left (and right) cell of $\cP$.
Therefore $\theta_x L(x)$ is indecomposable and has simple top $L(x)$,
see \cite[Subsection~4.5]{MM1}.
By self-duality, it also has simple socle $L(x)$. From the Kazhdan-Lusztig
combinatorics we get that, as a graded module,  $\theta_x L(x)$ is 
concentrated between degrees $\pm \ell(w_0^{\mathfrak{p}})$ with simple top and socle 
in the respective extremal degrees. As $\mathcal{O}_0$ is Koszul, it follows that 
$\theta_x L(x)$ has Loewy length $2\ell(w_0^{\mathfrak{p}})+1$.
The additive closure of $\theta_x L(x)$ carries, by \cite[Subsection~4.5]{MM1},
the natural structure of a $2$-representation of $\cP_x$ which is equivalent to
the cell $2$-representation of $\cP_x$ corresponding to the cell $\{\theta_x\}$.

Now consider the case of general $y$. From the Kazhdan-Lusztig
combinatorics we get that, as a graded module, $\theta_x L(y)$ is 
concentrated between degrees $\pm \ell(w_0^{\mathfrak{p}})$ and both extremal degrees
are one-dimensional.  As $\mathcal{O}_0$ is Koszul, it follows that the Loewy length
of $\theta_x L(y)$ is at most $2\ell(w_0^{\mathfrak{p}})+1$.

Assume that $\theta_x L(y)$ decomposes. Then, by 
Proposition~\ref{prop65}\eqref{prop65.1}, the module $\theta_x L(y)$ is a direct sum of 
several copies of the same indecomposable module $N$. 
We claim that the Loewy length of $N$ is strictly smaller than $2\ell(w_0^{\mathfrak{p}})+1$.
By the previous paragraph, the Loewy length of $N$ is at most $2\ell(w_0^{\mathfrak{p}})+1$.
If the Loewy length of $N$ were $2\ell(w_0^{\mathfrak{p}})+1$, then every submodule of codimension
one in $N\oplus N$ would also have Loewy length $2\ell(w_0^{\mathfrak{p}})+1$. At the
same time, the submodule
\begin{displaymath}
\bigoplus_{i=-\ell(w_0^{\mathfrak{p}})+1}^{\ell(w_0^{\mathfrak{p}})} \big(\theta_x L(y)\big)_i
\end{displaymath}
has Loewy length at most $2\ell(w_0^{\mathfrak{p}})$, a contradiction.

The above means that both $N$ and $\theta_x L(y)$ have Loewy length at most 
$2\ell(w_0^{\mathfrak{p}})$. At the same time, the additive closure of 
$N$ carries the structure of a $2$-representation of $\cP_x$ whose unique simple transitive quotient is the
cell $2$-representation of $\cP_x$ corresponding to the cell $\{\theta_x\}$.
Therefore some quotient of $N$ must have Loewy length $2\ell(w_0^{\mathfrak{p}})+1$ by the first paragraph
of this proof to provide an equivalence with the cell $2$-representation. 
This is a contradiction which completes the proof of this corollary.
\end{proof}

\section{Equivalences}\label{s5}

\subsection{Serre subquotient categories}\label{SecSerre}

For an arbitrary abelian category $\mathcal{C}$, a non-empty full subcategory 
$\mathcal{B}$ of $\mathcal{C}$ is called a {\em Serre subcategory} provided that,  for every short exact sequence
\begin{displaymath}
0\to Y_1\to X\to Y_2\to 0
\end{displaymath}
in $\mathcal{C}$, we have $X\in \mathcal{B}$ if and only if~$Y_1,Y_2\in \mathcal{B}$.

For a Serre subcategory $\mathcal{B}\subset \mathcal{C}$, we have the {\em Serre quotient category} 
$\mathcal{C}/\mathcal{B}$, which is defined as follows:
\begin{itemize}
\item the objects of $\mathcal{C}/\mathcal{B}$ are those of $\mathcal{C}$;
\item for any $X,Y\in \mathcal{C}$, we have
\begin{displaymath}
{\mathrm{Hom}}_{\mathcal{C}/\mathcal{B}}(X,Y)\;:=\;\varinjlim {\mathrm{Hom}}_{\mathcal{C}}(X',Y/Y'), 
\end{displaymath}
where $X'$, resp. $Y'$, runs over all sub-objects in $\mathcal{C}$ (ordered by inclusion) 
of $X$, resp. $Y$, such that~$X/X'\in\mathcal{B}$, resp. $Y'\in \mathcal{B}$.
\end{itemize}
We have the corresponding exact functor $\pi:\mathcal{C}\to\mathcal{C}/\mathcal{B}$, 
which is the identity on objects and maps a morphism $f:X\to Y$ to the corresponding element in the direct limit.

Assume now that~$\mathcal{C}=A$-mod, for $A$ a finite dimensional, associative and unital algebra
over an algebraically closed field $\Bbbk$. Any Serre subcategory 
$\mathcal{B}$ is then of the form $A/(AeA)$-mod, for some idempotent $e\in A$. The corresponding 
Serre quotient $\mathcal{C}/\mathcal{B}$ is then equivalent to~$eAe$-mod. In this case the 
functor $\pi$ is given by multiplication with $e$. In particular, we find that~$\pi$ yields 
isomorphisms
\begin{displaymath}
\pi:\;\,{\mathrm{Hom}}_{\mathcal{C}}(P,Q)\;\tilde{\to}\; {\mathrm{Hom}}_{\mathcal{C}/\mathcal{B}}(\pi P, \pi Q),\qquad\mbox{for $P,Q\in \add(Ae)$.} 
\end{displaymath}

Now consider a right exact functor $F$ on~$\mathcal{C}=A$-mod, which, up to isomorphism, is of the form 
$X\otimes_A-$ for some $A$-$A$-bimodule $X$. Assume that~$F$ restricts to a
functor on the Serre subcategory corresponding to the idempotent $e$ as above, that is $eX(1-e)=0$. 
Consequently, there is a right exact endofunctor $\overline{F}\cong eXe\otimes_{eAe}-$ of the 
Serre subquotient $eAe$-mod induced by~$F$. In other words, the following diagram commutes up to an isomorphism
of functors:
\begin{equation}\label{inducFun}
\xymatrix{\mathcal{C}\ar[d]^{\pi}\ar[rrr]^{F}&&&\mathcal{C}\ar[d]^{\pi}\\
\mathcal{C}/\mathcal{B}\ar[rrr]^{\overline{F}}&&&\mathcal{C}/\mathcal{B}}.
\end{equation}

For any right exact functor $F$ on~$\mathcal{C}=A$-mod, which preserves Serre subcategories
$\mathcal{A}\subset\mathcal{B}\subset\mathcal{C}$, we will refer to the right exact 
endofunctor $\overline{F}$ of $\mathcal{B}/\mathcal{A}$ obtained by restriction of $F$ 
to~$\mathcal{B}$ followed by the above procedure simply as the {\em functor induced from $F$.}

\subsection{Equivalences intertwining twisting functors}\label{s5.2}

Let $\lambda\in \Lambda$ be $W^{\mathfrak{p}}$-dominant. We define the sets $\mathsf{I}_\lambda$ and 
$\mathsf{J}_\lambda$ by
\begin{displaymath}
\mathsf{I}_\lambda\;=\;\{\mu\in\Lambda\,|\,\mu\preceq \lambda\}\;=\; 
W^{\mathfrak{p}}\cdot\lambda \;\amalg\; \mathsf{J}_\lambda. 
\end{displaymath}
We consider the Serre subquotient 
$\mathcal{A}_\lambda=\mathcal{O}^{\mathsf{I}_\lambda}/\mathcal{O}^{\mathsf{J}_\lambda}$ 
and use the notation of~\eqref{inducFun} for induced functors on Serre quotients.

\begin{theorem}\label{Thm1}
There exists an equivalence of categories
\begin{displaymath}
\Psi: \mathcal{O}_{\lambda}(\mathfrak{l})\;\tilde\to\; \mathcal{A}_\lambda,\quad\mbox{with}\quad\Psi(\Delta_{\mathfrak{l}}(\mu))\cong \Delta(\mu),\quad\mbox{for all $\mu\in W^{\mathfrak{p}}\cdot\lambda$,}
\end{displaymath}
such that the following diagram commutes, for any $w\in W^{\mathfrak{p}}$, up to isomorphism of functors:
\begin{displaymath}
\xymatrix{\mathcal{O}_\lambda(\mathfrak{l})
\ar[rd]_{\Psi}\ar[ddd]_{T_w}\ar[rr]^{{\rm Ind}^{\mathfrak{g}}_{\mathfrak{p}}}&&
\ar[dl]^{\pi}\mathcal{O}^{\mathsf{I}_\lambda}\ar[ddd]^{T_w}\\
&\mathcal{A}_\lambda\ar[d]^{\overline{T}_w}&\\
&\mathcal{A}_\lambda&\\
\mathcal{O}_\lambda(\mathfrak{l})\ar[ru]^{\Psi}
\ar[rr]_{{\rm Ind}^{\mathfrak{g}}_{\mathfrak{p}}}&&\ar[lu]_{\pi}\mathcal{O}^{\mathsf{I}_\lambda}
} 
\end{displaymath}
Here ${\rm Ind}^{\mathfrak{g}}_{\mathfrak{p}}$ is the parabolic induction functor which is full and faithful.
\end{theorem}

We start with the following lemma.

\begin{lemma}
The sets $\mathsf{I}_\lambda$ and $\mathsf{J}_\lambda$ are saturated.
\end{lemma}

\begin{proof}
For $\mathsf{I}_\lambda$, this is by construction. The claim for $\mathsf{J}_\lambda$ is 
equivalent to the claim that~$W^{\mathfrak{p}}\cdot\lambda$ forms an interval for $\preceq$. 
It is easy to see that for any $\mu_1\preceq \mu_2$, we have 
$\mu_1(H_{\mathfrak{l}})\le \mu_2(H_{\mathfrak{l}})$ with equality holding only if $\mu_1$ 
and $\mu_2$ are in the same $W^{\mathfrak{p}}$-orbit. This implies that any $W^{\mathfrak{p}}$-orbit is an interval.
\end{proof}

\begin{lemma}\label{TwistPreserves}
For any $y\in W^{\mathfrak{p}}$, the twisting functor $T_y$ on~$\mathcal{O}_\Lambda$ 
restricts to an endofunctor of both 
$\mathcal{O}^{\mathsf{I}_\lambda}$ and $\mathcal{O}^{\mathsf{J}_\lambda}$.
\end{lemma}

\begin{proof}
By induction, we can restrict to the case where $y=s$ is a simple reflection in $W^{\mathfrak{p}}$. 
Since the sets $\mathsf{I}_\lambda$ and $\mathsf{J}_\lambda$ are saturated and $T_s$ is right exact, 
it suffices to show that~$T_s\Delta(\mu)$ is in $\mathcal{O}^{\mathsf{I}_\lambda}$, 
resp. $\mathcal{O}^{\mathsf{J}_\lambda}$, for all $\mu\in \mathsf{I}_\lambda$, resp. 
$\mu\in\mathsf{J}_\lambda$.

For any $\mu\in \mathsf{I}_\lambda$, equation \eqref{VermaBGG} yields a short exact sequence
\begin{displaymath}
0\to \Delta(\mu)\to \Delta(\lambda)\to N\to 0,
\end{displaymath}
leading, by \eqref{twistVerma}, to the exact sequence
\begin{displaymath}
\cL_1T_sN\to T_s\Delta(\mu)\to \Delta(s\cdot\lambda). 
\end{displaymath}
The left term in is $\mathcal{O}^{\mathsf{I}_\lambda}$ since $\cL_1T_s$ is the functor taking 
the maximal $s$-finite submodule of $N$, see \cite[Proposition~5.10]{CM1} or the adjoint of 
\cite[Theorem~2.2]{simple}, so $\cL_1 T_sN\subset N\in \mathcal{O}^{\mathsf{I}_\lambda}$. 
Since the right term is also in $\mathcal{O}^{\mathsf{I}_\lambda}$, we find 
$T_s\Delta(\mu)\in \mathcal{O}^{\mathsf{I}_\lambda}$. This concludes the proof for $\mathsf{I}_\lambda$.

By equation~\eqref{twistVerma2}, the above actually proves that the set $\mathsf{I}_\lambda$ 
is closed under the left $W^{\mathfrak{p}}$-action. By definition, this implies that also 
$\mathsf{J}_\lambda$ is closed under the left $W^{\mathfrak{p}}$-action. Using 
equation~\eqref{twistVerma2} again thus proves that~$T_s\Delta(\mu)\in \mathcal{O}^{\mathsf{J}_\lambda}$, 
for any $\mu\in \mathsf{J}_\lambda$, which concludes the proof.
\end{proof}

\begin{proof}[Proof of Theorem~\ref{Thm1}]
We set $\mathsf{I}=\mathsf{I}_\lambda$ and $\mathsf{J}=\mathsf{J}_\lambda$. We use parabolic 
induction~${\rm Ind}^\mathfrak{g}_{\mathfrak{p}}$ from $\mathfrak{l}$-modules (interpreted as 
$\mathfrak{p}$-modules with trivial action of $\mathfrak{u}^+$) to~$\mathfrak{g}$-modules, which yields a functor
\begin{displaymath}
{\rm Ind}^{\mathfrak{g}}_{\mathfrak{p}}:\mathcal{O}_{\lambda}(\mathfrak{l})\to\mathcal{O}^{\mathsf{I}}. 
\end{displaymath}
This functor is full and faithful. Indeed, for any $N_1,N_2\in \mathcal{O}_{\lambda}(\mathfrak{l})$, 
by adjunction, we have
\begin{displaymath}
{\mathrm{Hom}}_{\mathcal{O}^{\mathsf{I}}}({\rm Ind}^{\mathfrak{g}}_{\mathfrak{p}} N_1,
{\rm Ind}^{\mathfrak{g}}_{\mathfrak{p}} N_2)\cong 
{\mathrm{Hom}}_{\mathcal{O}_{\lambda}(\mathfrak{l})}(N_1,
{\rm Res}^{\mathfrak{g}}_{\mathfrak{p}} {\rm Ind}^{\mathfrak{g}}_{\mathfrak{p}} N_2)\cong 
{\mathrm{Hom}}_{\mathcal{O}_{\lambda}(\mathfrak{l})}(N_1,N_2)
\end{displaymath}
since ${\rm Res}^{\mathfrak{g}}_{\mathfrak{p}} {\rm Ind}^{\mathfrak{g}}_{\mathfrak{p}} N_2\cong N_2\oplus X$
with ${\mathrm{Hom}}_{\mathcal{O}_{\lambda}(\mathfrak{l})}(N_1,X)=0$, where the latter follows by applying $\mathrm{ad}_{H_{\mathfrak{l}}}$.

Similarly we see that ${\rm Ind}^{\mathfrak{g}}_{\mathfrak{p}}$ is left adjoint to the composition
of ${\rm Res}^{\mathfrak{g}}_{\mathfrak{p}}$ followed by taking the correct $\mathrm{ad}_{H_{\mathfrak{l}}}$-eigenspace.
As the latter functor is exact, it follows that ${\rm Ind}^{\mathfrak{g}}_{\mathfrak{p}}$ sends projectives
to projectives. The easy observation
$[{\rm Res}^{\mathfrak{g}}_{\mathfrak{l}}L(\nu):L_{\mathfrak{l}}(\mu)]=\delta_{\nu\mu}$ implies
that all projectives in $\mathcal{O}^{\mathsf{I}}$ belong to the essential image of 
${\rm Ind}^{\mathfrak{g}}_{\mathfrak{p}}$. As ${\rm Ind}^{\mathfrak{g}}_{\mathfrak{p}}$ is full and faithful,
it is an equivalence of categories. The statement on Verma modules follows from the definition of 
parabolic induction.

By \eqref{TwistInd}, we have
\begin{displaymath}
T_y\circ{\rm Ind}^{\mathfrak{g}}_{\mathfrak{p}}\;\cong\; {\rm Ind}^{\mathfrak{g}}_{\mathfrak{p}}\circ T_y,
\end{displaymath}
for any $y\in W^{\mathfrak{p}}$. It thus suffices to show that~$T_y$ restricts 
to~$\mathcal{O}^{\mathsf{I}}$ and~$\mathcal{O}^{\mathsf{J}}$, which is done in Lemma~\ref{TwistPreserves}.
\end{proof}

\begin{corollary}\label{CorMult}
For any $x\in {}^{\mathfrak{p}}X$, $y\in X^{\mathfrak{p}}$ and $w_1,w_2\in W^{\mathfrak{p}}$, we have
\begin{displaymath}
[\Delta(w_1x):L(w_2x)]\;=\;[\Delta(w_1):L(w_2)]\;=\; [\Delta(yw_1):L(yw_2)].
\end{displaymath}
\end{corollary}

\begin{proof}
By Theorem~\ref{Thm1}, for $\lambda=x\cdot0$, we have
\begin{displaymath}
[\Delta(w_1x):L(w_2x)]\;=\;[\Delta_{\mathfrak{l}}(w_1x\cdot0):L_{\mathfrak{l}}(w_2x\cdot 0)]. 
\end{displaymath}
By equivalence of regular integral blocks in $\mathcal{O}_{\mathfrak{l}}$, see 
e.g.~\cite[Theorem~7.8]{Humphreys}, this number is equal 
to~$[\Delta_{\mathfrak{l}}(w_1\cdot0):L(w_2\cdot 0)]$. Applying Theorem~\ref{Thm1}, for 
$\lambda=0$, then yields the first equality in the proposition. The second can be 
obtained from the first one and the equality
\begin{displaymath}
[\Delta(u):L(v)]\;=\;[\Delta(u^{-1}):L(v^{-1})].
\end{displaymath}
The above is a well-known property of Kazhdan-Lusztig combinatorics. A direct proof is sketched below.

Consider the equivalence $F:\mathcal{O}_0\to {}^{\infty}_0H^1_0$ of \cite[Theorem~5.9]{BG}, 
where the  ${}^{\infty}_0H^1_0$ denotes the category of Harish-Chandra bimodules which admit 
generalized central character $\chi_0$ on the left and central character $\chi_0$ on the right, 
with $\chi_0$ being the central character of the trivial module.

By \cite[Satz~6.34]{Jantzen}, $F(L(v))$ and $F(L(v^{-1}))$ are linked by the duality on 
${}^{\infty}_0H^\infty_0$, or ${}^{1}_0H^1_0$, which exchanges the left and right action. 
Furthermore the description of $\Delta(v)$ as the quotient of $P(v)$ with respect to the 
submodule corresponding to all images of $P(x)\to P(v)$, for $x\prec v$, implies that 
$F(\Delta(v))$ and $F(\Delta(v^{-1}))$ are linked by the same duality.
\end{proof}

\begin{lemma}\label{LemProj1}
Consider the equivalence $\Psi:\mathcal{O}_0(\mathfrak{l})\;\tilde\to\; \mathcal{A}_0$ 
of Theorem~\ref{Thm1}, for the case $\lambda=0$. For any $w\in W^{\mathfrak{p}}$, the diagram of functors
\begin{displaymath}
\xymatrix{
\mathcal{O}_0(\mathfrak{l})\ar[r]^{\Psi}\ar[d]_{\theta_w}&
\mathcal{A}_0\ar[d]^{\overline{\theta}_w}&\ar[l]_{\pi}\mathcal{O}_0\ar[d]^{\theta_w}\\
\mathcal{O}_0(\mathfrak{l})\ar[r]_{\Psi}&\mathcal{A}_0&\ar[l]^{\pi}\mathcal{O}_0
} 
\end{displaymath}
extending~\eqref{inducFun}, commutes up to isomorphism.
\end{lemma}

\begin{proof}
That $\overline{\theta}_w$ is well-defined, is a special case of Lemma~\ref{LemTheta} below. 
This proves existence of the right-hand side of the diagram.

To study projective functors on $\mathcal{O}_0(\mathfrak{l})$, it suffices to consider 
$-\otimes V_0$ for simple finite dimensional $\mathfrak{l}$-modules $V_0$ on which 
$H_{\mathfrak{l}}$ acts trivially. For any such simple module, we have the corresponding 
simple $\mathfrak{g}$-module $V$ with same highest weight.

Recall that we have $\Psi=\pi\circ{\rm Ind}^{\mathfrak{g}}_{\mathfrak{p}}$. We can then calculate
\begin{displaymath}
\pi({\rm Ind}^{\mathfrak{g}}_{\mathfrak{p}}(M\otimes V_0))\;\cong\; 
\pi({\rm Ind}^{\mathfrak{g}}_{\mathfrak{p}}(M\otimes 
{\rm Res}^{\mathfrak{g}}_{\mathfrak{p}}V))\;\cong\; \pi({\rm Ind}^{\mathfrak{g}}_{\mathfrak{p}}(M)\otimes V).
\end{displaymath}
This shows that $\Psi\circ \theta_w\cong \overline{F}\circ \Psi$, for some projective functor 
$F$ on $\mathcal{O}_0(\mathfrak{g})$.
The identification of projective functors then follows immediately from equation~\eqref{DefTheta} 
and the observation $\Psi(P_{\mathfrak{l}}(x))\cong P(x)$, for all $x\in W^{\mathfrak{p}}$.
\end{proof}

\subsection{Equivalences intertwining projective functors}\label{s5.3}

Now we fix an element $z\in X^{\mathfrak{p}}$. We define the sets 
$\mathsf{K}_z$ and $\mathsf{L}_z$ by
\begin{displaymath}
\mathsf{K}_z\;=\;\{x\in W\,|\,z\preceq x\}\;=\; zW^{\mathfrak{p}}\,\amalg\, \mathsf{L}_z. 
\end{displaymath}
Then we define the Serre subquotient
\begin{displaymath}
\mathcal{B}_z\;=\;\mathcal{O}^{\mathsf{K}_z}/\mathcal{O}^{\mathsf{L}_z}. 
\end{displaymath}

\begin{theorem}\label{2dequiv}
There exists an equivalence
\begin{displaymath}
\Phi: \mathcal{O}_{0}(\mathfrak{l})\;\tilde\to\; \mathcal{B}_z,
\quad\mbox{ with}
\quad\Phi(\Delta_{\mathfrak{l}}(y))\cong \Delta(zy),\quad\mbox{for all $y\in W^{\mathfrak{p}}$.} 
\end{displaymath}
Furthermore, for any $x\in W^{\mathfrak{p}}$, the diagram of functors
\begin{displaymath}
\xymatrix{
\mathcal{O}_0(\mathfrak{l})\ar[r]^{\Phi}\ar[d]_{\theta_x}&
\mathcal{B}_z\ar[d]^{\overline{\theta}_x}&
\ar[l]_{\pi}\mathcal{O}_0^{\mathsf{K}_z}\ar[d]^{\theta_x}\\
\mathcal{O}_0(\mathfrak{l})\ar[r]_{\Phi}&\mathcal{B}_z&\ar[l]^{\pi}\mathcal{O}_0^{\mathsf{K}_z}
} 
\end{displaymath}
extending~\eqref{inducFun}, commutes up to isomorphism.
\end{theorem}

We start with the following lemma.

\begin{lemma}\label{LemTheta}
For any $x\in W^{\mathfrak{p}}$, the functor $\theta_x$ restricts to an 
endofunctor of both $\mathcal{O}^{\mathsf{K}_z}$ and $\mathcal{O}^{\mathsf{L}_z}$.
\end{lemma}

\begin{proof}
It suffices to consider $x=s$, a simple reflection in $W^{\mathfrak{p}}$.
Since the sets $\mathsf{K}_z$ and $\mathsf{L}_z$ are saturated, 
we can just prove that~$\theta_s\Delta(y)$ 
is in $\mathcal{O}^{\mathsf{K}_z}$, resp. $\mathcal{O}^{\mathsf{L}_z}$, for any 
$y\in \mathsf{K}_z$, resp. $y\in\mathsf{L}_z$. Since $\theta_s\Delta(y)$ is an 
extension of $\Delta(ys)$ and $\Delta(y)$ by \eqref{eqtheDel}, this is 
equivalent to showing that the sets $\mathsf{K}$ and $\mathsf{L}$ are 
closed under right $W^{\mathfrak{p}}$-multiplication. The latter follows 
from the easy observation that
\begin{displaymath}
\mathsf{K}\;=\;\coprod_{u\in X^{\mathfrak{p}}\,|\, z\preceq u} uW^{\mathfrak{p}}.
\end{displaymath}
The claim follows.
\end{proof}

\begin{lemma}\label{LemPC}
For any~$x\in W^{\mathfrak{p}}$, the projective cover of $L(zx)$ in $\mathcal{O}^{\mathsf{K}_z}$ satisfies
\begin{displaymath}
P^{\mathsf{K}_z}(zx)\cong T_z P(x)\cong \theta_x \Delta(z). 
\end{displaymath}
\end{lemma}

\begin{proof}
By \cite[Theorem~3.11]{Humphreys}, we have a short exact sequence
\begin{displaymath}
0\to M\to P(z)\to \Delta(z)\to 0, 
\end{displaymath}
where $M$ has a $\Delta$-flag with all Verma modules appearing of the form~$\Delta(y)$ 
with $y\prec z$, so, in particular, $y\not\in\mathsf{K}_z$. It thus follows 
by~\eqref{twistVerma} that
\begin{displaymath}
P^{\mathsf{K}}(z)\cong \Delta(z)\cong  T_z P(e). 
\end{displaymath}
Consequently, by Equations~\eqref{DefTheta} and \eqref{commProj}, we have
\begin{displaymath}
\theta_x P^{\mathsf{K}}(z)\cong T_zP(x), 
\end{displaymath}
for any $x\in W^{\mathfrak{p}}$. By adjunction, the module $\theta_x P^{\mathsf{K}}(z)$ 
is again projective. We claim that~$\theta_xP^{\mathsf{K}}(z)\cong P^{\mathsf{K}}(zx)$. 
To show this we first observe that the top of the module  
$\theta_x P^{\mathsf{K}}(z)\cong \theta_x\Delta(z)$ consists of simple modules of the 
form $L(zx')$ with $x'\in W^{\mathfrak{p}}$, by \eqref{eqtheDel}. Furthermore, 
since $\theta_x P^{\mathsf{K}}(z)\cong T_zP(x),$ for any $x'\in W^{\mathfrak{p}}$, we have
\begin{eqnarray*}
\dim{\mathrm{Hom}}_{\mathcal{O}^{\mathsf{K}}}(\theta_x P^{\mathsf{K}}(z),\Delta(zx'))&=
&\dim {\mathrm{Hom}}_{\mathcal{O}}(T_zP(x),T_z\Delta(x'))\\
&=&\dim {\mathrm{Hom}}_{\mathcal{O}}(P(x),\Delta(x'))\\
&=&[\Delta(zx'):L(zx)].
\end{eqnarray*}
Here, we used equations \eqref{twistVerma}, \eqref{isomPT} and Corollary~\ref{CorMult}.
\end{proof}

\begin{proposition}\label{PropEq}
There exists an equivalence of abelian categories
\begin{displaymath}
\Sigma:\mathcal{B}_e\to \mathcal{B}_z,\quad\mbox{with}\quad 
\Sigma(\Delta(y))\cong\Delta(zy),\mbox{ for all $y\in W^{\mathfrak{p}}$,} 
\end{displaymath}
such, for all $x\in W^{\mathfrak{p}}$, we have a diagram
\begin{displaymath}
\xymatrix{
\mathcal{B}_e\ar[rr]^{\overline{\theta}_x}\ar[d]_{\Sigma}&&\mathcal{B}_e\ar[d]^{\Sigma}\\
\mathcal{B}_z\ar[rr]^{\overline{\theta}_x}&&\mathcal{B}_z
} 
\end{displaymath}
with $\overline{\theta}_x$ induced from $\theta_x$, which commutes up to isomorphism.
\end{proposition}

\begin{proof}
Let $\mathcal{P}_0$ denote the full additive subcategory of $\mathcal{O}_0$ consisting of direct sums of modules isomorphic to elements of $\{P(y)\,,\,y\in W^{\mathfrak{p}}\}$. Similarly $\mathcal{P}_z$, is the full additive subcategory of $\mathcal{O}^{\mathsf{K}}_0$ consisting of direct sums of modules isomorphic to elements of $\{P^{\mathsf{K}}(zy)\,,\,y\in W^{\mathfrak{p}}\}$.

By Lemma~\ref{LemPC} and equation~\eqref{commProj}, we have a commuting diagram
\begin{displaymath}
\xymatrix{
\mathcal{P}_0\ar[rr]^{{\theta}_x}\ar[d]_{T_z}&&\mathcal{P}_0\ar[d]^{T_z}\\
\mathcal{P}_z\ar[rr]^{{\theta}_x}&&\mathcal{P}_z
}
\end{displaymath}
Now $\mathcal{P}_0$, resp. $\mathcal{P}_z$, is equivalent to the category of projective modules 
in $\mathcal{B}_e$, resp $\mathcal{B}_z$. Under that equivalence, $\theta_x$ is interchanged 
with $\overline{\theta}_x$. We now let $\Sigma:\mathcal{B}_e\to \mathcal{B}_z$ denote the 
equivalence corresponding to the equivalence between the categories of projective 
modules induced by $T_z$. By construction this admits the commuting diagram as in the formulation.

Consider the exact sequence
\begin{equation}\label{rightseq}
P\to P(y)\to \Delta(y)\to 0,
\end{equation}
in $\mathcal{O}_0$, where $P$ is a direct sum of $P(y')$, with $y'\in W^{\mathfrak{p}}$. By \eqref{twistVerma}, applying~$T_z$ to \eqref{rightseq} yields an exact sequence
\begin{displaymath}
T_zP\to T_zP(y)\to \Delta(zy)\to 0, 
\end{displaymath}
which is in $\mathcal{O}^{\mathsf{K}_z}$, by Lemma~\ref{LemPC}.
We can interpret \eqref{rightseq} as an exact sequence in $\mathcal{B}_e$, by applying the 
exact functor $\pi$, which leads to an exact sequence
\begin{displaymath}
\Sigma(P)\to \Sigma (P(y))\to \Sigma(\Delta(y))\to 0 
\end{displaymath}
in $\mathcal{O}^{\mathsf{K}_z}/\mathcal{O}^{\mathsf{L}_z}$. Comparing both exact sequences 
using Lemma~\ref{LemPC} yields the isomorphism $\Sigma(\Delta(y))\cong \Delta(zy)$.
\end{proof}

\begin{proof}[Proof of Theorem~\ref{2dequiv}.]
This is a combination of the equivalences in Lemma~\ref{LemProj1} and Proposition~\ref{PropEq}.
\end{proof}

\begin{remark}\label{remthickO}
The equivalences in Theorems~\ref{Thm1} and  \ref{2dequiv} extend to the thick category
$\mathcal{O}$ as in \cite{So3}. After extension, the two equivalences
are related by the duality on thick category $\mathcal{O}$ coming from the duality on
Harish-Chandra bimodules used in the proof of Lemma~\ref{CorMult}.
\end{remark}

\begin{remark}\label{remeqshuff}
The equivalences in Theorems~\ref{Thm1} also intertwines the corresponding shuffling functors.
\end{remark}


\subsection{Application to Conjecture~\ref{conj1}}\label{s5.4}

\begin{corollary}\label{cor71}
Let $\mathfrak{p}$ be a parabolic subalgebra of $\mathfrak{g}$. Let $x,y\in W^{\mathfrak{p}}$
and $z\in X^{\mathfrak{p}}$. Then $\theta_x L(y)$
is indecomposable in category $\mathcal{O}$ for $\mathfrak{l}$ if and only if
$\theta_x L(zy)$ is indecomposable in category $\mathcal{O}$ for $\mathfrak{g}$.
\end{corollary}

\begin{proof}
We have that $\theta_x L_{\mathfrak{l}}(y)$ is indecomposable if and only if 
$\theta_xL(zy)$ is indecomposable in $\mathcal{B}_z$ by Theorem~\ref{2dequiv}. 
From Equation~\eqref{eqtheDel}, we have that $\theta_x\Delta(zy)$ 
has a Verma flag with highest weights in $zW^{\mathfrak{p}}$, so the top of 
$\theta_x\Delta(zy)$ are simples of the form $L(zw)$, where $w\in W^{\mathfrak{p}}$. 
The top (and also socle) of $\theta_xL(zy)$ is a submodule of that for
$\theta_x\Delta(zy)$. Therefore the top (and the socle) 
of $\theta_xL(zy)$ consists only of simple modules not in $\mathcal{O}_0^{\mathsf{L}_z}$. 
Consequently, the endomorphism algebra of $\theta_xL(zy)$ in $\mathcal{O}_0$ 
is the same as it is in $\mathcal{B}_z$, so we find that 
$\theta_xL(zy)$ is indecomposable in $\mathcal{B}_z$
if and only if it is indecomposable in $\mathcal{O}_0$.
\end{proof}

As an immediate consequence of Corollary~\ref{cor71}, we obtain:

\begin{corollary}\label{cor72}
Conjecture~\ref{conj1} is true if and only if it is true under the additional assumption that
$\mathrm{supp}(x)=S$.
\hfill$\qed$
\end{corollary}

\section{Conjecture~\ref{conj1} for small values of $n$}\label{s8}

\subsection{The case $n=2$}\label{s8.2}

For $n=2$, we have $W=\{e,s\}$ and all elements in $W$ are of the form $w_0^{\mathfrak{p}}$, for some
$\mathfrak{p}$. Therefore Conjecture~\ref{conj1} is true in this case due to Subsection~\ref{s3.3}.

\subsection{The case $n=3$}\label{s8.3}

For $n=3$, the Weyl group $W=\{e,s,t,st,ts,w_0\}$ contains four Kazhdan-Lusztig right cells:
\begin{displaymath}
\{e\},\qquad
\{s,st\},\qquad
\{t,ts\},\qquad
\{w_0\}
\end{displaymath}
and $I(W)=\{e,s,t,w_0\}$. Note that all involutions in $W$ are of the form $w_0^{\mathfrak{p}}$, for some
$\mathfrak{p}$. Therefore Conjecture~\ref{conj1} is true in this case due to Subsections~\ref{s3.2}
and \ref{s3.3}.

\subsection{The case $n=4$}\label{s8.4}

For $n=4$, two-sided cells in $W\cong S_4$ are indexed by partitions of $4$:
\begin{displaymath}
(4),\quad (3,1), \quad (2,2), \quad (2,1,1)\quad (1,1,1,1).
\end{displaymath}
The two-sided cell corresponding to some partition $\lambda$ will be denote $\mathcal{J}_\lambda$.
We have that $\theta_x L(y)\neq 0$ implies $x\leq_{\mathcal{J}}y$, see \eqref{neweq123}.
By \cite[Theorem~5.1]{Ge},  the two-sided order on $W$ coincides with the opposite of the
dominance order on partitions.

By Corollary~\ref{cor22}, we may assume $x\in I(W)$ and $y\in I'(W)$. If $x\in \mathcal{J}_{(4)}$
or $x\in \mathcal{J}_{(3,1)}$, then $x=w_0^{\mathfrak{p}}$, for some
$\mathfrak{p}$. If $y\in \mathcal{J}_{(1,1,1,1)}$
or $y\in \mathcal{J}_{(2,1,1)}$, then $y=w_0^{\mathfrak{p}}w_0$, for some
$\mathfrak{p}$. If $x,y\in \mathcal{J}_{(2,2)}$, then $x\sim_{\mathcal{J}}y$.
Therefore Conjecture~\ref{conj1} is true in this case due to Subsection~\ref{s3.3}.

\subsection{The case $n=5$}\label{s8.5}

For $n=5$, two-sided cells in $W\cong S_5$ are indexed by partitions of $5$:
\begin{displaymath}
(5),\quad (4,1), \quad (3,2), \quad (3,1^2)\quad (2^2,1),\quad (2,1^3)\quad (1^5).
\end{displaymath}
The two-sided order coincides with the opposite of the
dominance order and is still linear (increasing from left to right).
We keep the conventions from the previous subsection.
By Corollary~\ref{cor22}, we may assume $x\in I(W)$ and $y\in I'(W)$. 

We have the following general result.

\begin{proposition}\label{propatleast8}
For any $n\geq 5$, Conjecture~\ref{conj1} is true if we have $x\in \mathcal{J}_{(n)}$, $x\in\mathcal{J}_{(n-1,1)}$ 
or $x\in\mathcal{J}_{(n-2,2)}$. Similarly, Conjecture~\ref{conj1} is true if $y\in \mathcal{J}_{(1^n)}$, $y\in\mathcal{J}_{(2,1^{n-2})}$ 
or $y\in\mathcal{J}_{(2^2,1^{n-4})}$.
\end{proposition}

\begin{proof}
The second statement follows from the first one using Subsection~\ref{s3.1}, so we only prove the
first statement. If $x\in \mathcal{J}_{(n)}$
or $x\in \mathcal{J}_{(n-1,1)}$, then $x=w_0^{\mathfrak{p}}$, for some
$\mathfrak{p}$. So, $\theta_x L(y)$ is either zero or indecomposable, for any $y$,
by Subsection~\ref{s3.3}.

The two-sided cell $\mathcal{J}_{(n-2,2)}$ has, by the hook formula, $\frac{n(n-3)}{2}$ right Kazhdan-Lusztig cells.
If $s$ and $t$ are two commuting simple reflections, then we have $st\in \mathcal{J}_{(n-2,2)}$ and it has the form
$w_0^{\mathfrak{p}}$, for some $\mathfrak{p}$. Therefore for such $x=st$, the module $\theta_x L(y)$ is either 
zero or indecomposable, for any $y$, by Subsection~\ref{s3.3}. There are $\frac{(n-2)(n-3)}{2}$ such elements.

Additionally, $\mathcal{J}_{(n-2,2)}$ contains $n-3$ involutions of the form $strs$, where $s,r,t$ are simple
reflections such that $s$ commutes neither with $r$ nor with $t$. As $n\geq 5$, 
these involutions have the property that 
not all simple reflections appear in their reduced expressions. Therefore Corollary~\ref{cor71} reduces 
these involutions to the case $n=4$ which is already treated above. 
As 
\begin{displaymath}
\frac{(n-2)(n-3)}{2}+(n-3)=\frac{n(n-3)}{2}, 
\end{displaymath}
the above covers all Kazhdan-Lusztig right cells in $\mathcal{J}_{(n-2,2)}$,
completing the proof.
\end{proof}

After Proposition~\ref{propatleast8}, it remains to consider the case 
$x,y\in \mathcal{J}_{(3,1^2)}$ which follows from Subsection~\ref{s3.3}.
Consequently, Conjecture~\ref{conj1} is true in the case $n=5$.

\subsection{The case $n=6$}\label{s8.6}

For $n=6$, the dominance order on partitions of $6$ is as follows:
\begin{displaymath}
\resizebox{12.5cm}{!}{
\xymatrix{
\\
&&&(3^2)\ar@{-}[rd]&&(2^3)\ar@{-}[rd]\\
(6)\ar@{-}[r]&(5,1)\ar@{-}[r]&(4,2)\ar@{-}[ru]\ar@{-}[rd]&&(3,2,1)\ar@{-}[ru]
\ar@{-}[rd]&&(2^2,1^2)\ar@{-}[r]&(2,1^4)\ar@{-}[r]&(1^6)\\
&&&(4,1^2)\ar@{-}[ru]&&(3,1^3)\ar@{-}[ru]
}
}
\end{displaymath}
This also gives the two-sided order on two-sided cells. In what follows, 
we simply write $\mathtt{i_1}\cdots\mathtt{i_k}$ for 
$s_{i_1}\cdots s_{i_k}$, where, as usual, $s_i$ denotes the transposition $(i,i+1)\in S_6$.

By Corollary~\ref{cor22}, we may assume $x\in I(W)$ and $y\in I'(W)$.
The cases $x\in \mathcal{J}_{(6)}$, $x\in \mathcal{J}_{(5,1)}$, $x\in \mathcal{J}_{(4,2)}$
and $y\in \mathcal{J}_{(1^6)}$, $x\in \mathcal{J}_{(2,1^5)}$, $x\in \mathcal{J}_{(2^2,1^4)}$
follow from Proposition~\ref{propatleast8}.

As $\theta_x L(y)$ is zero unless $x\leq_{\mathcal{J}}y$ (see \eqref{neweq123}),
the above allows us to restrict the index of the two-sided cells of 
$x$ and $y$ to the list $(3^2)$, $(4,1^2)$, $(3,2,1)$, $(2^3)$, $(3,1^3)$. 
We now give a full list of all those pairs which cannot be dealt with by direct applications 
of the results in Section~\ref{s3} and Corollary~\ref{cor71}, 
(here $x\in \mathcal{J}_{(3^2)}$ and $y\in \mathcal{J}_{(3,2,1)}
\cup\mathcal{J}_{(2^3)}\cup\mathcal{J}_{(3,1^3)}$, or $x\in \mathcal{J}_{(4,1^2)}$ and
$y\in \mathcal{J}_{(3,2,1)}\cup\mathcal{J}_{(3,1^3)}$):
\begin{enumerate}[$($I$)$]
\item\label{list1} $(\mathtt{1}\mathtt{2}\mathtt{3}\mathtt{4}\mathtt{5}\mathtt{4}\mathtt{3}\mathtt{2}\mathtt{1},
\mathtt{5}\mathtt{2}\mathtt{3}\mathtt{4}\mathtt{2}\mathtt{3}\mathtt{1})$, 
\item\label{list2} $(\mathtt{1}\mathtt{2}\mathtt{3}\mathtt{4}\mathtt{5}\mathtt{4}\mathtt{3}\mathtt{2}\mathtt{1},
\mathtt{3}\mathtt{4}\mathtt{5}\mathtt{1}\mathtt{2}\mathtt{3}\mathtt{2})$, 
\item\label{list3} $(\mathtt{1}\mathtt{2}\mathtt{3}\mathtt{4}\mathtt{5}\mathtt{4}\mathtt{3}\mathtt{2}\mathtt{1},
\mathtt{2}\mathtt{3}\mathtt{2}\mathtt{4}\mathtt{3}\mathtt{2})$, 
\item\label{list4} $(\mathtt{4}\mathtt{5}\mathtt{2}\mathtt{3}\mathtt{4}\mathtt{1}\mathtt{2},
\mathtt{5}\mathtt{2}\mathtt{3}\mathtt{4}\mathtt{2}\mathtt{3}\mathtt{1})$, 
\item\label{list5} $(\mathtt{4}\mathtt{5}\mathtt{2}\mathtt{3}\mathtt{4}\mathtt{1}\mathtt{2},
\mathtt{3}\mathtt{4}\mathtt{5}\mathtt{1}\mathtt{2}\mathtt{3}\mathtt{2})$, 
\item\label{list6} $(\mathtt{4}\mathtt{5}\mathtt{2}\mathtt{3}\mathtt{4}\mathtt{1}\mathtt{2},
\mathtt{2}\mathtt{3}\mathtt{2}\mathtt{4}\mathtt{3}\mathtt{2})$, 
\item\label{list7} $(\mathtt{4}\mathtt{5}\mathtt{2}\mathtt{3}\mathtt{4}\mathtt{1}\mathtt{2},
\mathtt{1}\mathtt{2}\mathtt{1}\mathtt{4}\mathtt{5}\mathtt{4})$, 
\item\label{list8} $(\mathtt{4}\mathtt{5}\mathtt{2}\mathtt{3}\mathtt{4}\mathtt{1}\mathtt{2},
\mathtt{3}\mathtt{4}\mathtt{5}\mathtt{4}\mathtt{1}\mathtt{2}\mathtt{3}\mathtt{1})$, 
\item\label{list9} $(\mathtt{3}\mathtt{4}\mathtt{5}\mathtt{2}\mathtt{3}\mathtt{4}\mathtt{1}\mathtt{2}\mathtt{3},
\mathtt{5}\mathtt{2}\mathtt{3}\mathtt{4}\mathtt{2}\mathtt{3}\mathtt{1})$, 
\item\label{list10} $(\mathtt{3}\mathtt{4}\mathtt{5}\mathtt{2}\mathtt{3}\mathtt{4}\mathtt{1}\mathtt{2}\mathtt{3},
\mathtt{3}\mathtt{4}\mathtt{5}\mathtt{1}\mathtt{2}\mathtt{3}\mathtt{2})$, 
\item\label{list11} $(\mathtt{3}\mathtt{4}\mathtt{5}\mathtt{2}\mathtt{3}\mathtt{4}\mathtt{1}\mathtt{2}\mathtt{3},
\mathtt{2}\mathtt{3}\mathtt{2}\mathtt{4}\mathtt{3}\mathtt{2})$, 
\item\label{list12} $(\mathtt{3}\mathtt{4}\mathtt{5}\mathtt{2}\mathtt{3}\mathtt{4}\mathtt{1}\mathtt{2}\mathtt{3},
\mathtt{1}\mathtt{2}\mathtt{1}\mathtt{4}\mathtt{5}\mathtt{4})$, 
\item\label{list13} $(\mathtt{3}\mathtt{4}\mathtt{5}\mathtt{2}\mathtt{3}\mathtt{4}\mathtt{1}\mathtt{2}\mathtt{3},
\mathtt{3}\mathtt{4}\mathtt{5}\mathtt{4}\mathtt{1}\mathtt{2}\mathtt{3}\mathtt{1})$. 
\end{enumerate}
We observe  that $\mathtt{5234231}\sim_{\mathcal{L}}\mathtt{1343}$
and that $L(\mathtt{1343})$ is annihilated by $\theta_{\mathtt{2}},\theta_{\mathtt{5}}$.
Therefore $\theta_x L(y)=0$ in \eqref{list1} and \eqref{list4}. Similarly, 
$\mathtt{3451232}\sim_{\mathcal{L}}\mathtt{2325}$ and 
$L(\mathtt{2325})$ is annihilated by $\theta_{\mathtt{1}},\theta_{\mathtt{4}}$.
Therefore $\theta_x L(y)=0$ in \eqref{list2} and \eqref{list5}. 
Further, $L(\mathtt{232432})$ is annihilated by $\theta_{\mathtt{1}},\theta_{\mathtt{5}}$ and 
$L(\mathtt{34541231})$ is annihilated by $\theta_{\mathtt{2}}$
and $L(\mathtt{121454})$ is annihilated by $\theta_{\mathtt{3}}$.
Therefore $\theta_x L(y)=0$ in \eqref{list3}, \eqref{list8} and \eqref{list12}. 
We also note that \eqref{list9} and \eqref{list10} differ by a symmetry
of the root system.

Choosing the shortest elements in the right cell of the first component and 
in the left cell of the second component and using Proposition~\ref{prop21},
the remaining cases \eqref{list6}, \eqref{list7}, \eqref{list9}, \eqref{list11} and \eqref{list13} become:
\begin{enumerate}[$($I$)$]
\setcounter{enumi}{13}
\item\label{list14} $(\mathtt{45231},\mathtt{232432})$;
\item\label{list15} $(\mathtt{45231},\mathtt{121454})$;
\item\label{list16} $(\mathtt{345231},\mathtt{1343})$;
\item\label{list17} $(\mathtt{345231},\mathtt{232432})$;
\item\label{list18} $(\mathtt{345231},\mathtt{1214543})$.
\end{enumerate}
Note that both $\mathtt{45231}$ and $\mathtt{345231}$ are short-braid hexagon avoiding and hence,
by \cite{BW}, we have:
\begin{equation}\label{eqn771}
\theta_{\mathtt{45231}}\cong\theta_{\mathtt{1}}\theta_{\mathtt{3}}\theta_{\mathtt{2}}
\theta_{\mathtt{5}} \theta_{\mathtt{4}},\qquad
\theta_{\mathtt{345231}}\cong\theta_{\mathtt{1}}\theta_{\mathtt{3}}
\theta_{\mathtt{2}} \theta_{\mathtt{5}} \theta_{\mathtt{4}}\theta_{\mathtt{3}}. 
\end{equation}

To proceed, we need the following statement.

\begin{proposition}\label{prop81}
{\hspace{2mm}}

\begin{enumerate}[$($a$)$]
\item\label{prop81.1} Let $s\in S$ and $y\in W$ be such that $ys\prec y$. Then the module $\theta_s L(y)$
has the following graded picture (the left column gives the degree): 
\begin{displaymath}
\begin{array}{r|ccc}
-1&&& L(y)\\
\hline
0&&& \displaystyle \bigoplus_{z\in W} L(z)^{\oplus m_{z}}\\
\hline
1& &&L(y),
\end{array}
\end{displaymath}
where 
\begin{itemize}
\item the degree $-1$ contains the simple top of $\theta_s L(y)$,
\item the degree $1$ contains the simple socle of $\theta_s L(y)$,
\item the degree $0$ contains what is known as the {\em Jantzen middle} of $\theta_s L(y)$.
\end{itemize}
Furthermore, $m_z\neq 0$, for $z\in W$, implies that $z\prec  zs$ and $ys\preceq z$.
Moreover, for each $z\in W$ such that $z\prec zs$, we have 
\begin{displaymath}
m_z=\dim\mathrm{Ext}^1(L(y),L(z))=
\begin{cases}
\mu(y,z), & y\prec z;\\
\mu(z,y), & z\prec y.
\end{cases}
\end{displaymath}
\item\label{prop81.2} 
Let $s,t\in S$ be such that $sts=tst$ and let $y\in W$ be such that $ys\preceq y$ and $yt\succeq y$.
Then 
\begin{displaymath}
\theta_{t}\theta_{s}L(y)\cong 
\begin{cases}
\theta_{t}L(ys), &  yst\prec ys;\\
\theta_{t}L(yt), &  \text{otherwise.}
\end{cases}
\end{displaymath}
\end{enumerate}
\end{proposition}

\begin{proof}
Claim~\eqref{prop81.1} is a well-known consequence of the Kazhdan-Lusztig conjecture.  
We just note that $\Delta(y)\tto L(y)$ and hence $\theta_s \Delta(y)\tto \theta_s L(y)$.
As $\theta_s \Delta(y)$ has $\Delta(ys)$ as a submodule and the corresponding quotient is
isomorphic to $\Delta(y)$ (which is, in turn, a submodule of $\Delta(ys)$), 
it follows that every simple subquotient of $\theta_s L(y)$
is a simple subquotient of $\Delta(ys)$. This justifies the relation $ys\preceq z$. The
$\mathrm{Ext}^1$-property follows from the observation that, if $M\in\mathcal{O}$ is
such that $M$ has simple top $L(y)$ and the kernel $K$ of $M\tto L(y)$ is killed by
$\theta_s$, then, by adjunction, $\theta_s L(y)\tto M$ and hence $K$ must be a submodule
of the Jantzen middle.

To prove claim~\eqref{prop81.2}, we note that $\theta_{t}L(y)=0$ by assumptions
and hence, due to  
claim~\eqref{prop81.1}, we just need to show that the Jantzen middle of $\theta_{s}L(y)$
contains a unique summand $L(z)$ such that $zt\prec z$, moreover, that 
this $z\in \{ys,yt\}$.

From  \cite[Fact~3.2]{Wa} it follows that a composition factor $L(z)$ of $\theta_{s}L(y)$,
which survives after $\theta_{t}$ must  satisfy $|\ell(z)-\ell(y)|=1$. If $z\prec y$, then 
from claim~\eqref{prop81.1} we have $ys\preceq z$ and hence $z=ys$. If $y\prec z$, then
$z=yt$ as $t$ is in the right descent set of  $z$ but not in the right descent set of $y$
and  \cite[Proposition~2.2.7]{BB} and~\cite[Corollary~2.2.5]{BB} give 
$y\preceq yt\preceq z$. This establishes $z\in \{ys,yt\}$.

Set $z=yt$ and $z'=ys$. To complete the proof, it now suffices to show that 
$zs\succeq z$ if and only if $z't\succeq z'$.
For the ``if" part, we note that $z't\succeq z'$ means that we have both
$yst\succeq ys$ and $y\succeq ys$.
By~\cite[Corollary~2.2.5]{BB} and~\cite[Corollary~2.2.8]{BB} 
we have $ysts,yt\succeq y, yst$. Once again, we obtain that $yts\succeq yt$.
The local Hasse diagram for the Bruhat order can, in this case, be depicted as follows:
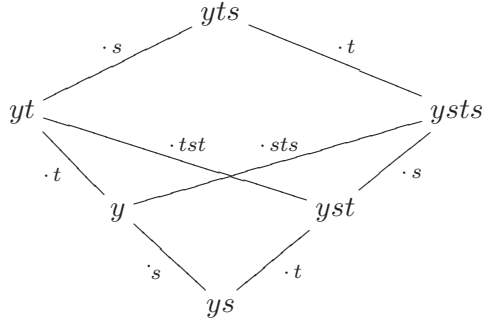
\begin{displaymath}
\xymatrix@R=1.8pc@C=1.8pc{
&&yts\ar@{-}[dll]_{\cdot\,s}\ar@{-}[drr]^{\cdot\,t}&&\\
yt\ar@{-}[dr]_{\cdot\,t}\ar@{-}[drrr]^{\cdot\,tst}&&&&ysts\ar@{-}[dl]^{\cdot\,s}
\ar@{-}[dlll]_{\cdot\,sts}\\
&y\ar@{-}[dr]_{\dot\,s}&&yst\ar@{-}[dl]^{\cdot\,t}&\\
&&ys&&
}
\end{displaymath}
For the ``only if" part, if $zs\succeq z$, then we have $ys\preceq y\preceq yt\preceq yts$.
Since $\ell(y)-1+2\geq \ell(ysts)=\ell(ytst)\geq \ell(y)+2-1$, we have $\ell(ysts)=\ell(y)+1$ which implies
$yst\succeq ys$.
\end{proof}

Proposition~\ref{prop81} has an interesting general consequence (compare with \cite[Fact~3.2]{Wa}).

\begin{corollary}\label{funnymu}
Assume that we are in type $A$.
Let $s,t$ be two simple reflections and $w\in W$ be such that $ws\succ w$ and $wt\prec w$. Then
\begin{equation}\label{eqfunny}
\sum_{\tilde{w}\in W:\tilde{w}s\prec \tilde{w}, \tilde{w}t\succ\tilde{w} }\mu(w,\tilde{w})\leq 1. 
\end{equation}
\end{corollary}

\begin{proof}
If $t$ and $s$ commute, then $ts$ is of the form $w_0^{\mathfrak{p}}$,
for some $\mathfrak{p}$. If $t$ and $s$ do not commute, then $tst=sts$ and
$\theta_s\theta_t L(w)$ is either indecomposable or zero by Proposition~\ref{prop81}\eqref{prop81.2}. 
On the other hand, by 
Proposition~\ref{prop81}\eqref{prop81.1}, the number of indecomposable 
direct summands of $\theta_s\theta_t L(w)$ 
is given by the left hand side of \eqref{eqfunny}. The claim follows.
\end{proof}

We note that the property in Corollary~\ref{funnymu} fails outside type $A$, in fact, already in
type $B_2$. This is the origin for the failure of Conjecture~\ref{conj1} in type  $B_2$. 

Now we will go through all the remaining cases \eqref{list14}--\eqref{list18}.

{\bf The case \eqref{list14}.}
We claim that Conjecture~\ref{conj1} is true for the pair \eqref{list14}.
Indeed, by Proposition~\ref{prop81}, we have 
\begin{displaymath}
\theta_{\mathtt{5}}\theta_{\mathtt{4}} L(\mathtt{232432})\cong  \theta_{\mathtt{5}}L(\mathtt{2324325})
\end{displaymath}
and the claim now follows from \eqref{eqn771}, Corollary~\ref{cor71} and the cases of smaller $n$ as the support of 
$\mathtt{5231}$ is not the whole of $S$.

{\bf The case \eqref{list16}.}
Next we claim that Conjecture~\ref{conj1} is true for the pair \eqref{list16}.
Indeed, by Proposition~\ref{prop81}, we have 
\begin{displaymath}
\theta_{\mathtt{2}}\theta_{\mathtt{3}} L(\mathtt{1343})\cong  \theta_{\mathtt{2}}L(\mathtt{13432})
\text{ and }\theta_{\mathtt{1}}\theta_{\mathtt{2}} L(\mathtt{13432})\cong  \theta_{\mathtt{1}}L(\mathtt{1343}).
\end{displaymath}
Taking into account \eqref{eqn771} and $\mathtt{345231}=\mathtt{321453}$, we again can 
reduce indecomposability for our pair $(x,y)$ 
to indecomposability for a new pair  $(x',y')$ in which the support of $x'$ is not the whole of $S$.
The claim now follows from Corollary~\ref{cor71} and the cases of smaller $n$.

{\bf The case \eqref{list18}.} From Proposition~\ref{prop81} we have that
\begin{displaymath}
\theta_{\mathtt{4}}\theta_{\mathtt{3}} L(\mathtt{1214543})\cong \theta_{\mathtt{4}} L(\mathtt{121454}). 
\end{displaymath}
Now, from \eqref{eqn771} it follows that 
\begin{displaymath}
\theta_{\mathtt{45231}}L(\mathtt{121454}) \cong 
\theta_{\mathtt{345231}}L(\mathtt{1214543}).
\end{displaymath}
Therefore Conjecture~\ref{conj1} is true for the pair 
\eqref{list18} if and only if Conjecture~\ref{conj1} is true for the pair 
\eqref{list15}.

{\bf The case \eqref{list15}.} We claim that the module $\theta_{\mathtt{42531}}L(\mathtt{121454})$
has simple top $L(\mathtt{1214543})$ and hence is indecomposable. This will establish 
Conjecture~\ref{conj1} for the pair \eqref{list15}. 

\begin{lemma}\label{lem85}
If $L(z)$ appears in the top of  $\theta_{\mathtt{42531}}L(\mathtt{121454})$, for some
$z\in W$, then $z=\mathtt{1214543}$.
\end{lemma}

\begin{proof}
If $L(z)$ appears in the top of  $\theta_{\mathtt{42531}}L(\mathtt{121454})$, for some
$z\in W$, then $z\sim_{\mathcal{R}}\mathtt{121454}$, see the proof of \cite[Theorem~6]{SHPO2}.
The right cell of $\mathtt{121454}$ consists of the following elements:
\begin{equation}\label{eq85-1}
\{\mathtt{121454},\mathtt{1214543},\mathtt{12143543},\mathtt{12145432},\mathtt{121435432}\}.
\end{equation}
Using  \eqref{eqn771}, by adjunction, we also have
\begin{displaymath}
\mathrm{Hom}(\theta_{\mathtt{42531}}L(\mathtt{121454}),L(z))\cong
\mathrm{Hom}(\theta_{\mathtt{42}}L(\mathtt{121454}),\theta_{\mathtt{531}}L(z)),
\end{displaymath}
which means that $\theta_{\mathtt{531}}L(z)\neq 0$. The only $z$ in the list \eqref{eq85-1}
for which $\theta_{\mathtt{531}}L(z))\neq 0$ is $z=\mathtt{1214543}$. The claim follows.
\end{proof}

After Lemma~\ref{lem85}, the case \eqref{list15} is completed by:

\begin{lemma}\label{lem86}
We have $\dim\mathrm{Hom}(\theta_{\mathtt{42531}}L(\mathtt{121454}),L(\mathtt{1214543}))=1$.
\end{lemma}

\begin{proof}
We use \eqref{eqn771} to write $\theta_{\mathtt{42531}}=\theta_{\mathtt{135}}\theta_{\mathtt{24}}$.
Hence, by adjunction, our claim is equivalent to 
\begin{displaymath}
\dim\mathrm{Hom}(\theta_{\mathtt{24}}L(\mathtt{121454}),\theta_{\mathtt{135}}L(\mathtt{1214543}))=1. 
\end{displaymath}
Note that the fact that this space is non-zero follows, by adjunction, from Lemma~\ref{lem85}
and the fact that $\theta_{\mathtt{42531}}L(\mathtt{121454})\neq 0$.

As $\mathtt{135}$ is of the form $w_0^{\mathfrak{p}}$, for some $\mathfrak{p}$, we can factorize
$\theta_{\mathtt{135}}=\theta_{\mathtt{135}}^{\mathrm{out}}\theta_{\mathtt{135}}^{\mathrm{on}}$
via translations on and out of the corresponding walls. Therefore, by adjunction, the above 
is equivalent to 
\begin{equation}\label{eq86-2}
\dim\mathrm{Hom}(\theta_{\mathtt{135}}^{\mathrm{on}}\theta_{\mathtt{24}}L(\mathtt{121454}),
\theta_{\mathtt{135}}^{\mathrm{on}}L(\mathtt{1214543}))=1. 
\end{equation}
The module $\theta_{\mathtt{135}}^{\mathrm{on}}L(\mathtt{1214543})$ is a simple module in a
singular block of $\mathcal{O}$, let us call it $L(\lambda)$. 

As $\mathtt{24}$ is of the form $w_0^{\mathfrak{q}}$, for some $\mathfrak{q}$, we know that
the graded module $\theta_{\mathtt{24}}L(\mathtt{121454})$ lives in degrees $\pm2,\pm1,0$.
Moreover, the degree $-2$ contains just its simple top $L(\mathtt{121454})$ and the
degree $2$ contains just its simple socle $L(\mathtt{121454})$. As $\mathtt{1214543}\succ \mathtt{121454}$,
we have  $\theta_{\mathtt{135}}^{\mathrm{on}}L(\mathtt{121454})=0$. This means that the module
$M:=\theta_{\mathtt{135}}^{\mathrm{on}}\theta_{\mathtt{24}}L(\mathtt{121454})$ lives in degrees
$\pm1,0$. 

Next we argue that $M$
can only have $L(\lambda)$ in the top (up to graded shift). Indeed, any occurrence of some other
$L(\mu)$ would, by adjunction, lead to a contradiction with Lemma~\ref{lem85}. This implies that the
degree $-1$ part of $M$ just consists of copies of $L(\lambda)$. In fact, there can only be a single
copy as $\dim\mathrm{Ext}^1(L(\mathtt{121454}),L(\mathtt{1214543}))=1$ 
(because $\mathtt{121454}$ and $\mathtt{1214543}$ differ by a single reflection) and hence appearance 
of $L(\mathtt{1214543})$ with multiplicity higher than one in degree $-1$ of $
\theta_{\mathtt{24}}L(\mathtt{121454})$ would contradict the
fact that $\theta_{\mathtt{24}}L(\mathtt{121454})$ has simple top.
Being a translation of 
a simple module, $M$ is self-dual, and hence the degree $1$ part of $M$ just consists of a 
single copy of $L(\lambda)$ as well. By the even-odd vanishing, $L(\lambda)$ cannot appear 
in degree $0$. 

The module $L(\mathtt{12134543})$ appears in the degree $0$ part 
of $\theta_{\mathtt{24}}L(\mathtt{121454})$ with non-zero multiplicity. As
$\theta_{\mathtt{135}}^{\mathrm{on}}L(\mathtt{12134543})\neq 0$, it follows that the
degree $0$ part of $M$ is, in fact, non-zero. By the previous paragraph, this degree $0$
part cannot contribute to the top of $M$. By self-duality, this degree $0$ part cannot
contribute to the socle of $M$ either. The latter, in turn, implies that the unique 
simple subquotient of the degree $1$ part cannot be in the top of the module.
This implies \eqref{eq86-2} and completes the proof.
\end{proof}

The arguments in the proof of Lemma~\ref{lem86} motivate the following question:

\begin{question}\label{quest87}
Let $n$ be arbitrary and define
\begin{displaymath}
\theta_{\mathrm{odd}}:=\theta_{\mathtt{1}}\theta_{\mathtt{3}}\theta_{\mathtt{5}}\cdots,\qquad
\theta_{\mathrm{ev}}:=\theta_{\mathtt{2}}\theta_{\mathtt{4}}\theta_{\mathtt{6}}\cdots.
\end{displaymath}
Is it true that, for all  $x,y\in W$, we have
\begin{displaymath}
\dim\mathrm{Hom}(\theta_{\mathrm{odd}}L(x),\theta_{\mathrm{ev}}L(y))\leq 1 ?
\end{displaymath}
\end{question}

We note that $\dim\mathrm{Hom}(\theta_{\mathrm{odd}}L(x),\theta_{\mathrm{ev}}L(y))>0$ is only possible
for $x,y\in W$ such that $x\sim_{\mathcal{R}}y$. For $x=y=w_{0}$, we have the following stronger statement.

\begin{proposition}\label{prop87-1}
Let $w\in W$ be such that each simple reflection appears at most once in a reduced expression of $w$.
Then $\theta_w L(w_0)$ has simple top. In particular, Question~\ref{quest87} has the positive answer,
for $x=w_0$ or $y=w_0$.
\end{proposition}

We note that the condition ``each simple reflection appears at most once in a reduced expression of $w$''
does not depend on the choice of a reduced expression as all reduced expressions can be obtained
from each other by applying braid relations, see \cite[Theorem~3.3.1]{BB}.

\begin{proof}
We start by noting that, for $z\in W$, the condition $w_0\sim_{\mathcal{R}}z$ implies $z=w_0$. Therefore
the second assertion of the proposition follows from the first one by adjunction. To prove the first one,
we first note that any $w$ as in the formulation is short-braid hexagon avoiding, cf. 
\cite[Equations~(9) and (10)]{BW}. Therefore we have the explicit form of the Kazhdan-Lusztig polynomial 
for such $w$, see \cite[Theorem~1]{BW} which implies that $\Delta(e)$ appears only once in the Verma 
flag of $P(w)$. By \cite[Theorem~8.1]{St} and \cite[Proposition~4]{FKM}, 
this means that $T(w_0w)\cong \theta_w L(w_0)$ 
has simple socle and thus also simple top, by self-duality. 
\end{proof}

{\bf The case \eqref{list17}.}
In this case we apply the same approach as in the previous case.

\begin{lemma}\label{lem87-2}
If $L(z)$ appears in the top of $\theta_{\mathtt{342531}}L(\mathtt{232432})$, for some $z\in W$,
then $z=\mathtt{23243215}$.
\end{lemma}

\begin{proof}
If $L(z)$ appears in the top of $\theta_{\mathtt{342531}}L(\mathtt{232432})$, for some $z\in W$, then
$z\sim_{\mathcal{R}}\mathtt{232432}$, see the proof of~\cite[Theorem~6]{SHPO2}. 
The right cell of $\mathtt{232432}$
consists of the following elements:
\begin{eqnarray}
&\{\mathtt{232432}, \mathtt{2324321}, \mathtt{2324325}, \mathtt{23214321}, \mathtt{23243215},
\mathtt{23243254},\label{eq87-3}\\
&\mathtt{213214321}, \mathtt{232432154}, \mathtt{232432543}, \mathtt{232143215}\}\nonumber.
\end{eqnarray}
Using~\eqref{eqn771}, by adjunction, we also have
\begin{displaymath}
\dim\mathrm{Hom}(\theta_{\mathtt{342531}}L(\mathtt{232432}),L(z))\cong\dim
\mathrm{Hom}(\theta_{\mathtt{342}}L(\mathtt{232432}), \theta_{\mathtt{531}}L(z)),
\end{displaymath}
which means that $\theta_{\mathtt{531}}L(z)\neq0$. The only $z$ in the list \eqref{eq87-3} for which
$\theta_{\mathtt{531}}L(z)\neq0$  is $z=\mathtt{23243215}$. The claim follows.
\end{proof}

After Lemma~\ref{lem87-2}, the case~\eqref{list17} is completed by:

\begin{lemma}\label{lem88-2}
We have $\dim\mathrm{Hom}(\theta_{\mathtt{342531}}L(\mathtt{232432}), L(\mathtt{23243215}))=1.$
\end{lemma}

\begin{proof}
Proof. We use~\eqref{eqn771} to write $\theta_{\mathtt{342531}}=\theta_{\mathtt{135}}\theta_{\mathtt{342}}$.
Hence, by adjunction, our
claim is equivalent to
\begin{displaymath}
\dim\mathrm{Hom}(\theta_{\mathtt{342}}L(\mathtt{232432}), \theta_{\mathtt{135}}L(\mathtt{23243215}))=1.
\end{displaymath}
Note that the fact that this space is non-zero is obvious.
As $\mathtt{135}$ is of the form $w^{\mathfrak{p}}_0$, for some $\mathfrak{p}$, 
we can factorize $\theta_{\mathtt{135}}= \theta_{\mathtt{135}}^{\mathrm{out}}\theta_{\mathtt{135}}^{\mathrm{on}}$ 
via translations on and out of the corresponding walls. Therefore, by adjunction,
the above is equivalent to
\begin{equation}\label{eq88-4}
\dim\mathrm{Hom}(\theta_{\mathtt{135}}^{\mathrm{on}}\theta_{\mathtt{342}}L(\mathtt{232432}), \theta_{\mathtt{135}}^{\mathrm{on}}L(\mathtt{23243215}))=1.
\end{equation}
The module $\theta_{\mathtt{135}}^{\mathrm{on}}L(\mathtt{23243215})$ is a simple module in a singular block of 
$\mathcal{O}$, let us call it $L(\lambda)$.

As $L(\mathtt{232432})$ is not annihilated by $\theta_{\mathtt{2}}, \theta_{\mathtt{3}}, \theta_{\mathtt{4}}$,
we know that the graded module
$\theta_{\mathtt{342}}L(\mathtt{232432})$ lives in degrees $\pm3,\pm2,\pm1,0$. 
Moreover, the degree $-3$ contains
just a copy of $L(\mathtt{232432})$ and the degree $3$ contains just a copy of
$L(\mathtt{232432})$. As $\mathtt{2324325}\succeq \mathtt{232432}$, we have $\theta_{\mathtt{135}}^{\mathrm{on}}L(\mathtt{232432})=0$. 

We want to prove that the graded module 
$M:=\theta_{\mathtt{135}}^{\mathrm{on}}\theta_{\mathtt{342}}L(\mathtt{232432})$ lives in degrees $\pm1,0$ 
with only $L(\lambda)$ in deg $\pm1$. We do it by an explicit computation.

From the table of Kazhdan-Lusztig polynomials for $S_6$, see~\cite{Go}, we have that
\begin{displaymath}
\{\mathtt{23432},\mathtt{2321432},\mathtt{243215432},\mathtt{213243254}\}
\end{displaymath}
is the set of all $\mathtt{3}$-finite permutations $z$ with non-zero $\mu$-functions
(whenever $\mu(z,\mathtt{232432})$ or $\mu(\mathtt{232432},z)$ makes sense).
Hence we have that the graded picture of $\theta_{\mathtt{3}}L(\mathtt{232432})$ is as follows:
\begin{displaymath}
\begin{array}{c||cccc}
-1 & L(\mathtt{232432})& & &\\
\hline
0 & L(\mathtt{23432})&L(\mathtt{2321432})& L(\mathtt{243215432})&L(\mathtt{213243254})\\
\hline
1 & L(\mathtt{232432})& &&
\end{array}
\end{displaymath}
We observe that all permutations
corresponding to simple subquotients of $\theta_{\mathtt{3}}L(\mathtt{232432})$ are $\mathtt{4}$-free and $\mathtt{5}$-finite.
Therefore, again, by Proposition~\ref{prop81}, we have
\begin{eqnarray*}
\theta_{\mathtt{5}}\theta_{\mathtt{4}}L(\mathtt{232432})&\cong & \theta_{\mathtt{5}}L(\mathtt{2324325}),\\
\theta_{\mathtt{5}}\theta_{\mathtt{4}}L(\mathtt{23432})&\cong & \theta_{\mathtt{5}}L(\mathtt{234325}),\\
\theta_{\mathtt{5}}\theta_{\mathtt{4}}L(\mathtt{2321432})&\cong & \theta_{\mathtt{5}}L(\mathtt{23214325}),\\
\theta_{\mathtt{5}}\theta_{\mathtt{4}}L(\mathtt{243215432})&\cong & \theta_{\mathtt{5}}L(\mathtt{2343215432}),\\
\theta_{\mathtt{5}}\theta_{\mathtt{4}}L(\mathtt{213243254})&\cong & \theta_{\mathtt{5}}L(\mathtt{21324325}).
\end{eqnarray*}

Observing that all simples on the right hand side of the above equations are $\mathtt{2}$-free
and $\mathtt{1}$-finite and $\mathtt{521}=\mathtt{215}$, we obtain
\begin{eqnarray*}
\theta_{\mathtt{1}}\theta_{\mathtt{2}}L(\mathtt{2324325})&\cong & \theta_{\mathtt{1}}L(\mathtt{23243215}),\\
\theta_{\mathtt{1}}\theta_{\mathtt{2}}L(\mathtt{234325})&\cong & \theta_{\mathtt{1}}L(\mathtt{2343215}),\\
\theta_{\mathtt{1}}\theta_{\mathtt{2}}L(\mathtt{23214325})&\cong & \theta_{\mathtt{1}}L(\mathtt{2321435}),\\
\theta_{\mathtt{1}}\theta_{\mathtt{2}}L(\mathtt{2343215432})&\cong & \theta_{\mathtt{1}}L(\mathtt{234321543}),\\
\theta_{\mathtt{1}}\theta_{\mathtt{2}}L(\mathtt{21324325})&\cong & \theta_{\mathtt{1}}L(\mathtt{213243215}).
\end{eqnarray*}
Note that the simple $L(\mathtt{2343215})$ is killed by $\theta_{\mathtt{3}}$.
For the case~\eqref{list17}, since $\mathtt{345231}=\mathtt{345213}$,
by tracking the degrees, we have that the graded picture of the module $\theta_xL(y)$ is obtained
by applying $\theta_{\mathtt{135}}$ to the following graded picture:
\begin{displaymath}
\begin{array}{c||cccc}
-1 & L(\mathtt{23243215})& & &\\
\hline
0 & L(\mathtt{2321435})&L(\mathtt{234321543})&L(\mathtt{213243215})&\\
\hline
1 & L(\mathtt{23243215})& &&
\end{array}
\end{displaymath}

Therefore the graded picture of  $M$ can be obtained by applying 
$\theta_{\mathtt{135}}^{\mathrm{on}}$ to the above graded picture, 
which implies that the graded module $M$ lives in degree $\pm1,0$ 
with $L(\lambda)$ appearing once in degree $1$ and once in degree $-1$.

The module $\theta_{\mathtt{135}}^{\mathrm{on}}L(\mathtt{2321435})\neq 0$ 
appears in the degree $0$ part of 
$M$ with non-zero multiplicity. By~Lemma~\ref{lem87-2}, this degree $0$
part cannot contribute to the top of $M$. By self-duality, this degree $0$ part
cannot contribute to the socle of $M$ either. The latter, in turn, implies that
the unique simple subquotient of the degree $1$ part cannot be in the top of
the module. This implies~\eqref{eq88-4} and completes the proof.
\end{proof}

Hence Conjecture~\ref{conj1} is true in the case $n=6$.

\section{Shuffling and twisting simple modules}\label{s9}

\subsection{Conjecture and results}\label{s9.1}

Given the close connection between shuffling and translation functors, it is natural to investigate whether shuffling simple modules yields indecomposable modules. This leads us to the following conjecture.

\begin{conjecture}\label{conjshuff}
For $\mathfrak{g}$ of any type and for all $x,y\in W$, the module $C_xL(y)$ is either indecomposable or zero.
\end{conjecture}

We summarize the evidence in favor of this conjecture in the following proposition. 
All indecomposability of modules is based on the observation that they have simple top or socle.

\begin{theorem}\label{PropShuff}
Let $\mathfrak{g}$ be of any type and consider $x,y\in W$.
\begin{enumerate}[$($a$)$]
\item\label{PropShuff.1} The functor $C_x$ on $\mathcal{O}_0$ is indecomposable. 
More precisely, its endomorphism algebra is isomorphic to the algebra $\mathtt{C}$ of coinvariants for 
the Weyl group $W$.
\item\label{PropShuff.2} The complex $\cL C_x L(y)$ is indecomposable in $\mathcal{D}^b(\mathcal{O}_0).$
\item\label{PropShuff.3} If $\theta_x L(y)$ has simple top, then either $C_xL(y)$ is zero or has simple top.
\item\label{PropShuff.4} The module $C_xL(y)$ has simple top in the following cases:
\begin{enumerate}[$($i$)$]
\item\label{PropShuff.4.1} $x=w_0^{\mathfrak{p}}$, for some parabolic subalgebra 
$\mathfrak{p}$, and $y$ is a longest representative in $W/W^{\mathfrak{p}}$ 
(equivalently, $x\le_{\mathcal{R}}y^{-1}$), in this case the simple top is $L(y)$;
\item\label{PropShuff.4.2} $y$ is a Duflo involution and $x\sim_{\mathcal{R}} y$, in this case the
simple top is $L(x)$.
\end{enumerate}
\item\label{PropShuff.5} If $y=w_0^{\mathfrak{p}}w_0$, for some parabolic subalgebra 
$\mathfrak{p}$, and $x\le_{\mathcal{R}}y^{-1}$ (equivalently, $w_0w_0^{\mathfrak{p}}x$ 
is a shortest representative in $W^{\mathfrak{p}}\backslash W$), then $C_x L(y)$ has simple socle $L(yx)$.
\item\label{PropShuff.6} If $C_xL(y)\not=0$, then $x\le_{\mathcal{R}} y^{-1}$.
\end{enumerate}
\end{theorem}

These results will be proved below. Despite the partial results in the above proposition, the 
following question seems unanswered in general. 

\begin{question} 
When is the module $C_xL(y)$ zero? 
\end{question}

We note that the example for $B_2$ where $\theta_xL(y)$ is decomposable, see \cite[Section~5.1]{KiM}, 
does not lead to a counterexample to Conjecture~\ref{conjshuff}.

\begin{example}
If $\mathfrak{g}=B_2$, then $C_{st}L(ts)$ is indecomposable, despite  
the fact that $\theta_{st}L(ts)$ is decomposable. The module  $C_{st}L(ts)$ is an 
indecomposable module admitting a short exact sequence
\begin{displaymath}
0\to L(e)\oplus L(ts)\to C_{st}L(ts)\to L(t)\oplus L(tst)\to 0.
\end{displaymath}
\end{example}

\subsection{Relation with twisting simple modules}\label{s9.2}

\begin{lemma}\label{EquivTC}
Let $\mathfrak{g}$ be of any type. For all $x,y\in W$, we have that $C_xL(y)$ is indecomposable if and only if $T_{x^{-1}}L(y^{-1})$ is indecomposable. Moreover,
\begin{displaymath}
\End_{\mathcal{O}}(C_xL(y))\cong \End_{\mathcal{O}}(T_{x^{-1}}L(y^{-1})). 
\end{displaymath}
\end{lemma}

\begin{proof}
Let $\widetilde{\mathcal{O}}_0$ be the full subcategory of $\mathcal{O}_0$ of modules 
which admit a central character. Under the equivalence $\widetilde{\mathcal{O}}_0\cong {}^1_0H^1_0$ 
of \cite{BG}, Soergel's autoequivalence of $ {}^1_0H^1_0$ from \cite{So0} which exchanges 
left and right action on a bimodule, maps $C_xL(y)$ to $T_{x^{-1}}L(y^{-1})$, by \cite[Satz~6.34]{Jantzen}. 
The claim follows.
\end{proof}

\begin{lemma}\label{lemlemk}
Let $\mathfrak{p}$ be a parabolic subalgebra of $\mathfrak{g}$ and take $x,y\in W^{\mathfrak{p}}$ 
and $z\in X^{\mathfrak{p}}$. If $C_xL(y)$ is indecomposable in category $\mathcal{O}$ for 
$\mathfrak{l}$, then $C_xL(zy)$ is indecomposable in category $\mathcal{O}$ for $\mathfrak{g}$.
\end{lemma}

\begin{proof}
It follows from Theorem~\ref{2dequiv} that, for any $x\in W^{\mathfrak{p}}$ and $\mu\in\mathfrak{h}^\ast$, we have 
\begin{displaymath}
\End_{\mathcal{O}(\mathfrak{l})}(C_xL(y))\;\cong\; 
\End_{\mathcal{B}_z}(C_xL(zy)). 
\end{displaymath}
Note that each simple constituent of 
the top of the module $C_xL(zy)$ survives, by construction, projection onto $\mathcal{B}_z$. 
From the definition of a Serre quotient we thus have that $\End_{\mathcal{O}(\mathfrak{g})}(C_xL(zy))$
is a subalgebra of $\End_{\mathcal{B}_z}(C_xL(zy))$. Therefore, if $\End_{\mathcal{B}_z}(C_xL(zy))$
is local, then so is $\End_{\mathcal{O}(\mathfrak{g})}(C_xL(zy))$.
The claim follows.
\end{proof}

Twisting functors, contrary to shuffling functors, can easily be defined in full generality 
for singular blocks as well.

\begin{question}
Is the module $T_xL$ indecomposable or zero, for any simple module $L$ in $\mathcal{O}$ and $x\in W$?
\end{question}

\subsection{On the (derived) shuffling functor}\label{s9.3}

\begin{proposition}
For all $x\in W$, there exists a linear complex of projective functors on $\mathcal{O}^{\mZ}_0$, 
which is isomorphic to $\cL C_x$ as a functor on $\mathcal{D}^b(\mathcal{O}_0^{\mZ})$. 
This complex has length $\ell(x)$ and its zero term is given by $\theta_x$.
\end{proposition}

\begin{proof}
Consider the endofunctor of $\mathcal{D}^b(\mathcal{O}_0^{\mZ})$ given by tensoring with 
the linear complex
\begin{displaymath}
0\to \Id\langle -1\rangle\to \theta_s\to 0 
\end{displaymath}
and the taking the total complex. Using \cite[Theorem~2]{simple}, we see that 
the action of this functor on projective modules coincides with the action of $\cL C_s$.
Consequently, these two functors are isomorphic. Now,
we can take a reduced expression of $x$, compose the complexes of the above form corresponding to the factors
of this reduced expression, and then form the total complex. This is, by construction, a complex 
$\mathbf{P}_x^{\bullet}$ of projective functors of length $\ell(x)$ which corresponds to the functor $\cL C_x$.

When we evaluate $\mathbf{P}_x^{\bullet}$ at $\Delta(e)$, we get a complex of projective modules which 
is quasi-isomorphic to $\cL C_x\Delta(e)\cong \Delta(x)$. Since $\mathcal{O}_0$ is standard 
Koszul in the sense of \cite{ADL}, see~\cite[Corollary~3.8]{ADL}
or \cite[Theorem~2.1]{Ma}, $\Delta(x)$ is quasi-isomorphic to a linear complex of
projective modules. Therefore the complex $\mathbf{P}_x^{\bullet}(\Delta(e))$ of projective modules 
is homotopic to a linear complex and thus is a direct sum of a linear part and a number of trivial 
complexes of the form 
\begin{displaymath}
0\to P\cong P\to 0.
\end{displaymath}
As all homomorphisms between projective modules in $\mathcal{O}_0$ are 
realizable via natural transformations between projective functors by \cite[Theorem~3.5]{BG}, it follows that 
$\mathbf{P}_x^{\bullet}$ is isomorphic to the direct sum of a linear complex and a number of trivial complexes. 
The claim follows.
\end{proof}

\begin{corollary}\label{Corlinear}
The functor $C_x$ on $\mathcal{O}_0^{\mZ}$ is isomorphic to the cokernel of a natural transformation
\begin{displaymath}
\theta\langle -1\rangle\to \theta_x, 
\end{displaymath}
for $\theta$ a direct sum of projective functors $\theta_z$, with $z\in W$, without any grading shift.
\end{corollary}

\begin{lemma}\label{EndShuff}
For any $x\in W$, the algebra $\End(C_x)$  is isomorphic to $\mathtt{C}$.
\end{lemma}

\begin{proof}
Consider endomorphism algebras of the functors $\Id$ and $C_x$ as objects in 
the category of $\mC$-linear additive functors on $\mathcal{O}_0$.
We have a commuting diagram of algebra morphisms
\begin{equation}\label{eqdiag1}
\xymatrix{
\End(\Id)\ar[rr]^{()_{C_x}}\ar[d]^{\Ev_{P(w_0)}}&&\End(C_x)\ar[d]^{\Ev_{P(w_0)}}\\
\End_{\mathcal{O}}(P(w_0))\ar[rr]^{C_x}&&\End_{\mathcal{O}}(P(w_0))
} 
\end{equation}
where the vertical arrows are evaluation of natural transformations on the module 
$P(w_0)$ and the upper horizontal arrow maps a natural transformation 
$\eta:\Id\to \Id$ to the natural transformation $\eta_{C_x}:\Id\circ C_x\to \Id\circ C_x$, 
defined by $(\eta_{C_x})_M=\eta_{C_xM}$, for all $M\in \mathcal{O}_0$. 
We also used the fact that $C_x P(w_0)\cong P(w_0)$ which follows, by induction on the length of 
$C_x$ from $C_s P(w_0)\cong P(w_0)$. The latter is checked by a direct computation directly
from the definitions. Indeed,
$\theta_s P(w_0)\cong P(w_0)\langle 1\rangle \oplus P(w_0)\langle -1\rangle$ and the adjunction
morphism $P(w_0)\langle -1\rangle\to \theta_s P(w_0)$ is injective as the simple socle of 
$P(w_0)$ is not killed by $\theta_s$, which yields $C_s P(w_0)\cong P(w_0)$.

It is proved in \cite[Section~2]{SoergelD} that $\End(\Id)\cong \mathtt{C}$ and that 
the left vertical arrow in \eqref{eqdiag1} is an isomorphism. That the lower horizontal arrow is an 
isomorphism follows from the fact that $\cL C_{x}$ is an auto-equivalence of 
$\mathcal{D}^b(\mathcal{O}_0)$. It thus suffices to prove that the right vertical arrow is injective.

Consider a short exact sequence
\begin{displaymath}
0\to \Delta(e)\to P(w_0)\to \mathrm{Coker}\to 0, 
\end{displaymath}
where $\mathrm{Coker}$ has a Verma flag. Applying all indecomposable projective functors 
to this sequence and  adding all this up gives a 
short exact sequence
\begin{displaymath}
0\to P\to Q\to N\to 0, 
\end{displaymath}
where $P$ is a projective generator of $\mathcal{O}_0$, the module $Q$ is a direct sum of copies of 
$P(w_0)$ and $N$ is a module with Verma flag. 
By Section~\ref{SecShuff}, we thus have a short exact sequence
\begin{displaymath}
0\to C_{x}P\to C_xQ\to C_{x}N\to 0.
\end{displaymath}
Evaluating an arbitrary $\eta\in \End(C_x)$ thus yields a commuting diagram with exact rows
\begin{displaymath}
\xymatrix{
0\ar[r]& C_xP\ar[r]\ar[d]^{\eta_P}&C_xQ\ar[d]^{\eta_Q}\\
0\ar[r]&C_xP\ar[r]&C_x Q
} 
\end{displaymath}
Since $P$ is projective generator, $\eta_P$ is not zero as soon as $\eta\neq 0$. 
It thus follows that $\eta_Q$ cannot be zero either in this case 
and the injectivity requested in the previous paragraph is proved.
\end{proof}

\subsection{Proof of Theorem~\ref{PropShuff}}\label{s9.4}

Claim~\eqref{PropShuff.1} is proved in Lemma~\ref{EndShuff}. Claim~\eqref{PropShuff.2} follows 
from the fact that $\cL C_x$ is an auto-equivalence of the bounded derived category. 
Claims~\eqref{PropShuff.3} and \eqref{PropShuff.6} follow from the epimorphism $\theta_x\tto C_x$ 
in  Corollary~\ref{Corlinear} and the criterion for vanishing of $\theta_xL(y)$ in 
\eqref{neweq123}.

In the following two lemmata we prove the remaining claims~\eqref{PropShuff.4} and \eqref{PropShuff.5}.

\begin{lemma}\label{llln1}
Let $\mathfrak{g}$ be of any type and consider $x,y\in W$.
\begin{enumerate}[$($a$)$]
\item  If $x=w_0^{\mathfrak{p}}$ for some parabolic subalgebra $\mathfrak{p}$, then
the module $C_xL(y)$ has simple top $L(y)$ if and only if $x\le_{\mathcal{R}}y^{-1}$ and is zero otherwise.
\item If $y$ is a Duflo involution and $x\sim_{\mathcal{R}} y$, then the module $C_xL(y)$ has simple top $L(x)$.
\end{enumerate}
\end{lemma}

\begin{proof}
Since $\theta_{w_0^{\mathfrak{p}}}L(y)$ either has simple top $L(y)$ or is zero, 
the corresponding property for $C_{w_0^{\mathfrak{p}}}L(y)$ follows from the paragraph above the lemma.
We have
\begin{displaymath}
C_{w_0}L(y)=\begin{cases}0&\mbox{if }\;\; y\not=w_0\\
\nabla(e)&  \mbox{if } \;\;y=w_0.\end{cases} 
\end{displaymath}
Indeed, the case $y=w_0$ follows from \cite[Proposition~5.12]{shuffling}. The case $y\not=w_0$
follows by combining the facts that $\theta_{w_0}L(y)=0$ and, further, that $\theta_{w_0}L(y)\tto C_{w_0}L(y)$,
where the latter follows from Corollary~\ref{Corlinear}.
That $C_{w_0^{\mathfrak{p}}}L(y)$ is non-zero if and only if $y$ is a longest 
representative in $W/W^{\mathfrak{p}}$ then follows from Theorem~\ref{2dequiv}. 
The latter condition is equivalent to $w_0^{\mathfrak{p}}\le_{\mathcal{R}}y^{-1}$. This proves the first claim.

Now assume that $y$ is a Duflo involution. This means that there exists a module 
$K\subset\Delta(e)$ such that $\theta_x\Delta(e)\cong \theta_xK$, for any 
$x\sim_{\mathcal{R}}y$, and the top of $K$ is $L(y)$, see e.g.~\cite[Proposition~17]{MM1}. 
Consequently, $P(x)\cong\theta_x K$ surjects onto $\theta_xL(y)$, so 
$\theta_xL(y)$ has simple top $L(x)$. By Corollary~\ref{Corlinear}, 
we have an exact sequence in $\mathcal{O}_0^{\mZ}$
\begin{displaymath}
\theta L(y)\langle 1\rangle\to \theta_x L(y)\to C_xL(y)\to 0. 
\end{displaymath}
The module in the middle term has its simple top in degree $-a(y)$, see \cite[Proposition~1(c)]{SHPO2}. 
It is also proved in {\it loc. cit.} that the term in the left-hand side 
is zero in degree $-a(y)$. 
This means, in particular, that $C_xL(y)\not=0$ and that it has simple top, which completes the proof.
\end{proof}

For the following lemma, we need to introduce some notation for the parabolic version 
of category $\mathcal{O}$. Let $\mathcal{O}^{\mathfrak{p}}$ denote the full subcategory 
of modules in $\mathcal{O}$ which are locally $U(\mathfrak{l})$-finite. The standard modules in $\mathcal{O}^{\mathfrak{p}}_0$ are given by the parabolic Verma modules $\Delta^{\mathfrak{p}}(x)$, 
where $x\in {}^{\mathfrak{p}}X$. They can either be defined as the maximal locally 
$U(\mathfrak{l})$-finite quotient modules of the ordinary Verma modules, or as the modules 
induced from simple finite dimensional $\mathfrak{p}$-modules. By construction, 
$\Delta^{\mathfrak{p}}(x)$ has simple top $L(x)$. We also write $\nabla^{\mathfrak{p}}(x)$ 
for the dual module of $\Delta^{\mathfrak{p}}(x)$.

\begin{lemma}\label{lemlll2}
For a parabolic subalgebra $\mathfrak{p}$ of $\mathfrak{g}$ and $x\in W$ a shortest 
representative in $W'\backslash W$, with $W'=w_0W^{\mathfrak{p}}w_0$, we have
\begin{displaymath}
C_xL(w_0^{\mathfrak{p}}w_0)\cong \nabla^{\mathfrak{p}}(w_0^{\mathfrak{p}}w_0x).
\end{displaymath}
Hence, $C_xL(w_0^{\mathfrak{p}}w_0)$ is non-zero and  has simple socle $L(w_0^{\mathfrak{p}}w_0x)$.
\end{lemma}

\begin{proof}
Consider $\lambda\in\Lambda^+$ with $W_\lambda=W'$. By \cite{BGS}, we have the Koszul duality functor
\begin{displaymath}
K:\mathcal{D}^b(\mathcal{O}_0^{\mathfrak{p}})^{\mZ}\;\tilde\to\;\mathcal{D}^b(\mathcal{O}_{\lambda})^{\mZ},
\end{displaymath}
normalized such that $\nabla^{\mathfrak{p}}(z)$ is mapped to $\Delta(z^{-1}w_0\cdot\lambda)$. Under this duality, $\cL C_w$ is exchanged with $\cL T_{w^{-1}}$, see \cite[Section~6.5]{MOS}. 
With $w_0^\lambda$, the longest element in $W'$, we thus have
\begin{displaymath}
C_xL(w_0^{\mathfrak{p}}w_0)\cong K^{-1}(T_{x^{-1}}\Delta(w_0^\lambda\cdot\lambda))\cong K^{-1}(\Delta(x^{-1}w_0^\lambda\cdot\lambda))\cong\nabla^{\mathfrak{p}}(w_0^{\mathfrak{p}}w_0x). 
\end{displaymath}
This completes the proof.
\end{proof}

\noindent
KC: School of Mathematics and Statistics, University of Sydney, Australia;
E-mail: {\tt kevin.coulembier@sydney.edu.au}
\vspace{2mm}

\noindent
VM: Department of Mathematics, University of Uppsala, Box 480, SE-75106, Uppsala, Sweden;
E-mail: {\tt  mazor@math.uu.se}

\noindent
XZ: Department of Mathematics, University of Uppsala, Box 480, SE-75106, Uppsala, Sweden;
E-mail: {\tt  xiaoting.zhang@math.uu.se}
\date{}

\end{document}